\documentclass[pdflatex,sn-mathphys-num]{sn-jnl}% Math and Physical Sciences Numbered Reference Style
\usepackage{graphicx}%
\usepackage{multirow}%
\usepackage{amsmath,amssymb,amsfonts}%
\usepackage{amsthm}%
\usepackage{mathrsfs}%
\usepackage[title]{appendix}%
\usepackage{xcolor}%
\usepackage{textcomp}%
\usepackage{manyfoot}%
\usepackage{booktabs}%
\usepackage{algorithm}%
\usepackage{algorithmicx}%
\usepackage{algpseudocode}%
\usepackage{listings}%
\usepackage{graphicx}
\usepackage{subcaption}

\theoremstyle{definition}
\newtheorem{theorem}{Theorem}[section]
\newtheorem{proposition}[theorem]{Proposition}
\newtheorem{lemma}[theorem]{Lemma}
\newtheorem{corollary}[theorem]{Corollary}
\theoremstyle{definition}
\newtheorem{definition}[theorem]{Definition}
\newtheorem{example}[theorem]{Example}
\theoremstyle{remark}
\newtheorem{remark}[theorem]{Remark}
\usepackage{mathtools}
\usepackage{mathrsfs}
\usepackage{amssymb}

\allowdisplaybreaks[1]
\begin{document}

\title[Convolution Operators on Weighted Hahn Spaces]{Convolution Operators on Weighted Hahn Spaces}

\author[1]{\fnm{Sayan} \sur{Saha}}\email{sayans@iitbhilai.ac.in}
\equalcont{These authors contributed equally to this work.}

\author*[1]{\fnm{Arnab} \sur{Patra}}\email{arnabp@iitbhilai.ac.in}

\equalcont{These authors contributed equally to this work.}


\affil[1]{\orgdiv{Department of Mathematics}, \orgname{Indian Institute of Technology Bhilai}, \orgaddress{\street{Kutelabhata}, \city{Durg}, \postcode{491002}, \state{Chhattisgarh}, \country{India}}}

%\affil[2]{\orgdiv{Department}, \orgname{Organization}, \orgaddress{\street{Street}, \city{City}, \postcode{10587}, \state{State}, \country{Country}}}

%\affil[3]{\orgdiv{Department}, \orgname{Organization}, \orgaddress{\street{Street}, \city{City}, \postcode{610101}, \state{State}, \country{Country}}}


\abstract{This paper studies the convolution operator on weighted Hahn sequence spaces. The boundedness and compactness of these operators, together with the multiplier algebras of the weighted Hahn space and its dual, are investigated. A complete characterization of the spectrum and fine spectrum is obtained, with illustrative examples. The introduction of the weighted framework leads to the emergence of new multiplier and spectral properties.

%This paper studies convolution operators on weighted Hahn sequence spaces. Necessary and sufficient conditions are established for their boundedness and compactness. The multiplier algebras f weighted Hahn space and its dual are analyzed. A complete characterization of the spectrum and fine spectrum is obtained, and illustrative examples are provided. The influence of the weighted framework on the associated multiplier algebras and spectral properties is analyzed.
}




\maketitle
%\tableofcontents
\section{Introduction}\label{sec1} 
A convolution operator is an operator that convolves two functions or arrays to produce  third function, which shows that how the shape of one function modifies the shape of the other function. In the setting of sequence spaces, a discrete convolution operator may be viewed as an operator that associates to two sequences the sequence of coefficients corresponding to the product of the associated power series.
In mathematical terms, if $x = (x_{n})$ and $y = (y_{n})$ are two complex sequences, indexed by $\mathbb{N}_0 = \mathbb{N} \cup \{0\}.$ Their convolution $x*y$ is defined by
\[(x*y) = \big(x_{0}y_{0}, x_{1}y_{0}+x_{0}y_{1}, x_{2}y_{0}+x_{1}y_{1}+x_{0}y_{2}, \cdots \big),\]
Equivalently, if the sequences $x$ and $y$ are viewed as the coefficients of formal power series or polynomials, then 
$x*y$ corresponds precisely to the sequence of coefficients of their product.

For a fixed sequence $a = (a_{n})$, with $a_{0}, a_{1} \neq 0$, the convolution operator $T_{a}$ on the linear sequence space $ \mathbb{C}^{\mathbb{N}_{0}}$ is defined by
\[
T_{a}(x) = a*x,\  x \in  \mathbb{C}^{\mathbb{N}_{0}}.
\]
The matrix representation of the convolution operator $T_{a}$ with respect to the standard ordered basis is
\begin{align*} [T_{a}]=
	\begin{bmatrix}
		a_{0} & 0 & 0 & 0 & \cdots\\
		a_{1} & a_{0} & 0 & 0 & \cdots\\
		a_{2} & a_{1} & a_{0} & 0 & \cdots\\
		a_{3} & a_{2} & a_{1} & a_{0} & \cdots\\
		\vdots & \vdots & \vdots & \vdots & \ddots
	\end{bmatrix}.
\end{align*}
For a linear subspace $X$ of $\mathbb{C}^{\mathbb{N}_{0}},$ two central problems in the study of convolution operators are to determine under which condition $T_a$ maps $X$ into itself continuously, and if this is the case, to identify the spectrum of $T_a$ acting on $X$.

The spectral theory of convolution operators has been extensively investigated in classical sequence spaces, particularly when the defining sequence $a$ has only finitely many nonzero entries. For instance, the spectral properties of the operator $B(r,s)$ corresponding to $a = (r, s, 0, 0, \cdots )$ has been studied by Altay and Ba\c{s}ar \cite{altay2005fine} over the sequence spaces $c_{0}$ and $c$; by Furkan et al. \cite{furkan2006fine} over the sequence spaces $\ell^{1}$ and $bv$; and by Bilgi\c{c} and Furkan \cite{bilgicc2008fine} over the sequence spaces $\ell^{p}$ and $bv_{p}$ $(1 < p < \infty)$. Extension to the three banded operator $B(r,s,t)$ with $a = (r,s,t,0,0,\cdots)$ were studied by Bilgi\c{c} and Furkan \cite{bilgicc2007fine} over the spaces $\ell^1 \ \text{and}\  bv$; by Furkan et al. \cite{furkan2010fine} over the sequence spaces $\ell^p \ \text{and}\ bv_{p}$ for $1<p<\infty$; by Furkan et al. \cite{furkan2007fine} over the sequence space $c_0$ and $c$. More generally, Birbonshi and Srivastava \cite{birbonshi2017some} have studied the fine spectrum of $n$ band triangular matrix with $a = (a_{0}, a_{1}, \cdots,a_{n}, 0, 0, \cdots)$ on the spaces $c_0, \ \ell^p, \ bv_p$ for $1<p< \infty$ while Patra et al. \cite{patra2019some} investigated the fine spectrum of its compact perturbations on $\ell^p$. Recently, the boundedness, compactness and the spectrum of the convolution operators acting over discrete Ces\`aro spaces \cite{ricker2019convolution}, dual Ces\`aro sequence spaces \cite{curbera2023convolution} and Fr\'echet sequence spaces \cite{ricker2019convolutionF} has been studied by various researchers.

With this connection, for $X \subset \mathbb{C}^{\mathbb{N}_{0}},$ another interesting problem is to study the collection of all such sequences $a$ such that the convolution $a*b \in X$ for all $b \in X.$ This collection is known as multiplier of $X$ and denoted by $\mathscr{M}(X).$ The multiplier of some of the classical sequence spaces are mentioned below:
\begin{enumerate}
    \item[(i)] $\mathscr{M}(\ell^1) = \mathscr{M}(\ell^{\infty}) = \ell^1$ \cite{nikol1966spaces},
    \item[(ii)] $\mathscr{M}(\ell^p) = \ell^q$ where $1/p+1/q = 1$ \cite{nikol1966spaces},
    \item[(iii)] $\mathscr{M}(\mbox{ces}_p) = \ell^1$ where $1 < p< \infty$ \cite{ricker2019convolution},
    \item[(iv)] $\mathscr{M}(\mathbb{C}^{\mathbb{N}}) = \mathbb{C}^{\mathbb{N}}$ \cite{ricker2019convolutionF},
\end{enumerate}
where $\mbox{ces}_p$ denote the sequence space generated by the Ces\`aro operator.

Our main interest lies in investigating the boundedness, compactness, and the spectral properties of the convolution operator over the weighted Hahn sequence spaces. In 1922, Hahn \cite{hahn1922folgen} introduced the Hahn sequence space $h$ defined by
\[
h = \Biggl\{x \in \mathbb{C}^{\mathbb{N}_{0}} : \sum_{n=0}^{\infty}(n+1) \cdot \big|x_{n} - x_{n+1}\big| < \infty, \mbox{ and } \lim_{n \to \infty} x_n = 0 \Biggr\}. 
\]
There has been an increasing interest to study the spectral properties of bounded linear operators, specially banded matrices defined over $h$. See the recent articles \cite{tuug2022generalized,tug2022domain,malkowsky2021roots,malkowsky2021compact}.

The remaining part of this paper is organised as follows. In Section~\ref{Section 2}, we recall the necessary preliminary definitions and auxiliary results that will be used throughout the paper. Section~\ref{Sect3} is devoted to the study of various structural properties of the weighted Hahn sequence space $h_d$ and the associated sequence space $bs_d,$ where $d$ denotes a weight vector and $bs_d$ is isomorphic to the dual of $h_d$. In Section~\ref{sect4}, we establish a complete description of the boundedness and compactness of convolution operators. In particular, Theorems~\ref{bdd-hrhs} and~\ref{cpt-hrhs} provide necessary and sufficient conditions for the boundedness and compactness, respectively, of the convolution operator $T_{a}: h_{r} \to h_{s},$ where $r$ and $s$ are (possibly distinct) weight vectors corresponding to the domain and codomain. It is further shown in Corollary~\ref{cpt-h} that the convolution operator $T_a$ on the Hahn space $h$ is never compact. With this connection, the multiplier algebras of the sequence spaces $h_d$ and $bs_d$ are discussed in Section \ref{sect5}. Section \ref{sect6} is devoted to study the spectrum and fine spectrum of $T_a \in \mathcal{B}(h_d).$ We provide a complete description of the spectrum (Corollary~\ref{s-Ta-hd}), point spectrum (Proposition~\ref{pts-Ta-hd}), residual spectrum (Theorem~\ref{rs-Ta-hd}), and continuous spectrum (Corollary~\ref{cs-Ta-hd}). Finally, we provide illustrative examples to demonstrate the applicability of the obtained results.

\section{Notations and Preliminaries}\label{Section 2}
Let $X$ and $Y$ be complex Banach spaces and $T:X \to Y$ be a bounded linear operator. $\mathcal{R}(T)$ denotes the range of $T$ and $\mathcal{N}(T)$ denotes the kernel of $T$. The set of all bounded linear operators $T: X \to Y$ is denoted by $\mathcal{B}(X,Y)$ with $\mathcal{B}(X,X) = \mathcal{B}(X).$ The operator norm of $T \in \mathcal{B}(X,Y)$ is denoted by $\|T\|_{\mathcal{B}(X,Y)}$ with $\|T\|_{\mathcal{B}(X,X)} = \|T\|_{\mathcal{B}(X)}.$ The adjoint of $T \in \mathcal{B}(X),$ denoted by $T^*$ is a bounded linear operator on the dual space $X^*$ of $X$, is defined by
\[(T^* \zeta)x = \zeta (Tx) \mbox{ for all } \zeta \in X^*, \ x \in X.\]

The spectrum of $T \in \mathcal{B}(X)$, denoted by $\sigma(T,X)$, is the collection of all complex numbers $\lambda$ such that $(T-\lambda I)$ is not invertible on $X$. The resolvent $\rho(T,X)$ of $T$ is the collection of all complex numbers $\lambda$ such that $\lambda \in \mathbb{C}\setminus \sigma(T,X).$ 
%\begin{align*}
%	\sigma(T,X) &=\{\lambda \in \mathbb{C}: (T - \lambda I) \ \text{is not invertible}\} \\
%	&= \{\lambda \in \mathbb{C}: \mathcal{N}(T)\neq \{0\}\ \text{or} \ \mathcal{R}(T)\neq X\},
%\end{align*}
%and
%\begin{align*}
%	\rho(T,X) &= \{(T - \lambda I)\ \text{has a bounded inverse}\}\\
%	&= \{\lambda \in \mathbb{C}: \mathcal{N}(T)= \{0\}\ \text{and} \ \mathcal{R}(T)= X\}.
%\end{align*}
For any point $\lambda \in \sigma(T,X)$ consider the following cases,
\begin{enumerate}
	\item[(i)] $\mathcal{N}(T- \lambda I)\neq \{0\}$,
	\item[(ii)] $\mathcal{N}(T- \lambda I)= \{0\}\ \text{but} \ \mathcal{R}(T- \lambda I)\neq X$.
\end{enumerate}
The collection of points $\lambda \in\mathbb{C}$ for which $\mathcal{N}(T - \lambda I)\neq \{0\}$ i.e., the operator $(T - \lambda I)$ is not injective
is called the point spectrum of $T$ and denoted by  $\sigma_{p}(T,X)$. The point spectrum of $T$
is precisely the set of all eigenvalues of $T$. The collection of points $\lambda \in \mathbb{C}$ such that, $\mathcal{N}(T - \lambda I)= \{0\}\ \text{and}\ \overline{\mathcal{R}(T-\lambda I)}= X$ but $\mathcal{R}(T-\lambda I) \neq X$ is called continuous spectrum of $T$ and denoted by $\sigma_{c}(T,X)$ and the collection of points $\lambda \in \mathbb{C}$ such that, $\mathcal{N}(T - \lambda I)= \{0\}\ \text{and}\ \overline{\mathcal{R}(T-\lambda I)}\neq X$ is called residual spectrum of $T$ and denoted by $\sigma_{r}(T,X)$. These three parts of the spectrum form a disjoint partition such that
\[\sigma(T,X)= \sigma_{p}(T,X)\cup \sigma_{c}(T,X)\cup \sigma_{r}(T,X). \]
For a bounded linear operator $T,\ \sigma(T,X)$ is always a non-empty, bounded, and closed subset of $\mathbb{C}$. The spectral radius of $T$ is denoted by $r(T)$, is defined by  \[r(T)=\sup\{|\lambda|: \lambda\in\sigma(T,X)\}.\]
Let $\mathcal{F}(T)$ denotes the family of all functions $f$ which are analytic on some neighborhood of $\sigma(T,X)$. Then we have the following theorem.
\begin{theorem}[\textbf{Spectral mapping theorem}]\cite[p.569]{dunford1988linear} \label{spt-th}
	If $f$ is in $\mathcal{F}(T)$, then \[f(\sigma(T,X))=\sigma(f(T),X).\]
\end{theorem}

%Throughout this paper, let $X,Y$  be a complex Banach spaces and $T$ be a bounded linear operator from $X$ to $Y$. The operator norm of $T$, denoted by $\|T\|_{\mathcal{B}(X,Y)}$, is defined as:
%\begin{align*}
	%\|T\|_{\mathcal{B}(X,Y)} = \sup\{\|T(x)\|:x\in X, \|x\|\leq1\}.
%\end{align*}
%Let $\mathcal{R}(T)$ denotes the range of $T$ and $\mathcal{N}(T)$ denotes the kernel of $T$, i.e.,
%\begin{align*}
%	\mathcal{R}(T) &= \{ T(x): x \in X\}, \\
%	\mathcal{N}(T) &= \{x \in X: T(x)=0\}.\ \ \ \ \ \ \ \ \ \ \ \ \ \
%\end{align*}
%The space of all bounded linear operators from the Banach space $X$ to the Banach space $Y$ is denoted by $\mathcal{B}(X,Y)$$\big(\mathcal{B}(X)\ \text{if}\ X=Y\big)$. The adjoint of $T$, denoted by $T^*$, is
%a bounded linear operator on the dual space $X^*$ of $X$ defined by 
%\begin{align*}
%	(T^*f)(x)=f(Tx),\ \ \ \ \text{for all}\ f \in X^*\ \  \text{and}\ x \in X.
%\end{align*}

Let $\mathbb{N}_0 = \mathbb{N} \cup \{0\}.$ All the sequences and vectors are indexed by $\mathbb{N}_0$. By a sequence space, we understand a linear subspace of space $\mathbb{C}^{\mathbb{N}_{0}}$, the space of all complex sequences. The space of all finitely non-zero sequences, null sequences, convergent sequences, bounded series, convergent series, absolutely $p$-summable $(1 \leq p < \infty)$ sequences and bounded sequences are denoted by $c_{00}, \ c_{0}, \ c, \ bs, \ cs, \ell^p (1 \leq p < \infty), \ \ell^\infty$ respectively.

A sequence space $X$ with a linear topology is called a $K$-space provided each of the maps $p_{n}:X \to \mathbb{C}$ defined by $p_{n}(x)=x_{n}$ is continuous for all $n\in \mathbb{N}_{0}$. A $K$-space $X$ is called an $FK$-space provided $X$ is a complete linear metric space. An $FK$-space whose topology is normable is called a $BK$-space. Given an $FK$-space $X \supset c_{00} $, we denote the $n$-th section of a sequence $x= (x_{k}) \in X $ by $x^{[n]} = \sum_{k=0}^{n}x_{k}e_{k}$, with $e_{k}=\big(0, \cdots, 0, 1, 0, \cdots\big),$ the sequence having $1$ in the $k$-th position for some $k \in \mathbb{N}_{0}$. Let $e$ denotes the vector whose each component is $1$, i.e., $e=(1,1,1,\cdots)$. We say that $x$ has the property 
\begin{enumerate}
	\item[(i)] $AK$ if $x^{[n]} \to x $ as $n \to \infty$, %(Abschnittskonvergenz),
	\item[(ii)] $AB$ if $\big(x^{[n]}\big)$ is bounded, %(Abschnittsbeschr\"anktheit),
	\item[(iii)]  $AD$ if $x \in \overline{c_{00}}$,\\ where $\overline{c_{00}}$ denotes the closure of $c_{00}.$ 
\end{enumerate}

The $\alpha$-, $\beta$- and $\gamma$-dual of a sequence space $X$, which are denoted by $ X^{\alpha},\ X^{\beta} \ \text{and} \ X^{\gamma}$, respectively, are defined by
\begin{enumerate}
	\item[(i)] $X^{\alpha} = \bigl\{x \in \mathbb{C}^{\mathbb{N}_{0}}: xy=(x_{n}y_{n}) \in \ell^{1},\ \text{for all} \ y \in X \bigr\}$,
	\item [(ii)] $X^{\beta} = \bigl\{x \in \mathbb{C}^{\mathbb{N}_{0}}: xy=(x_{n}y_{n}) \in cs,\ \text{for all} \ y \in X \bigr\}$ and
	\item [(iii)] $X^{\gamma} = \bigl\{x \in \mathbb{C}^{\mathbb{N}_{0}}: xy=(x_{n}y_{n}) \in bs,\ \text{for all} \ y \in X \bigr\}$.
\end{enumerate}
The $\alpha$-, $\beta$-, and $\gamma$-duals of a sequence space $X$ are also known as the K{\"o}the-Toeplitz dual, generalized K{\"o}the-Toeplitz dual, and Garling dual, respectively.
\begin{lemma} \cite[Theorem 7.2.9]{wilansky2000summability} \label{b-c-dual}
	Let $ X \supset c_{00}$ be an $FK$-space.
	\begin{enumerate}
		\item[(i)] Then there is a linear one-to-one map $\ \ \widehat{} : X^{\beta} \to X^{*} $.
		\item[(ii)] If $ X $ has $ AK $, then the map  $\ \  \widehat{} : X^{\beta} \to X^{*} $ is an isomorphism.
	\end{enumerate}
\end{lemma}
%The $\alpha$-, $\beta$- and $\gamma$-dual of $c_{00}$ are all equal to $\mathbb{C}^{\mathbb{N}_{0}}$; while the $\alpha$-, $\beta$- and $\gamma$-dual of $c_{0}, c \ \text{and}\ \ell^\infty $ are $\ell^{1}$; for the space $ \ell^{1}$ these are all equal to $\ell^{\infty}$.

%\textcolor{blue}{
%Let $X \subseteq \mathbb{C}^{\mathbb{N}_{0}}$ be any sequence space. The sequence $a = (a_{n}) $ defines a multiplier on $X$ if the convolution $ a*x \in X $ for every $x \in X $ and the collection of all such $ a $ is denoted by $\mathscr{M}(X)$, i.e., 
%\[
%\mathscr{M}(X) = \Bigl\{a \in \mathbb{C}^{\mathbb{N}_{0}} : a*x \in X \ \text{for all} \ x \in X \Bigr\}.
%\]
%From \cite{nikol1966spaces} it is known that 
%\begin{enumerate}
%	\item [(i)] $\ell^1 \subsetneq \mathscr{M}(\ell^p) \subsetneq \ell^p$, for $1<p<\infty$;
%	\item[(ii)] $ \mathscr{M}(\ell^p) = \mathscr{M}(\ell^{p\prime})$, for $1/p + 1/p{\prime} = 1$;
%	\item[(iii)] $\mathscr{M}(\ell^{p_{1}}) \subsetneq \mathscr{M}(\ell^{p_{2}})$, for $1 \leq p_{1} < p_{2} \leq 2$.
%\end{enumerate}
%}

Let $T: X \to Y$ be a bounded linear operator, where $X$ and $Y$ are Banach spaces with Schauder bases. If $T$ admits a matrix representation, we denote it by $[T]$. For $n, k \in \mathbb{N}_0$, we write $[T]_{n,*}$ and $[T]_{*,k}$ for the $n$-th row and $k$-th column of $[T]$, respectively.

Finally, let $A = \big(a_{nk}\big)$ be an infinite matrix and $X$ and $Y$ be any two sequence spaces. If $Ax$ exists and is in $Y$ for every sequence $x=\big(x_{n}\big) \in X$, then we say that $A$ defines a matrix mapping from $X$ into $Y$, and we denote it by writing $A:X \to Y$. By $(X:Y)$ we denote the class of all matrices $A$ such that $A:X \to Y $. Thus, $A\in (X:Y)$ if and only if $( Ax )_{n}= \sum_{k=0}^{\infty}a_{nk}x_{k}$ is finite and $Ax = \big( (Ax)_{n}\big) \in Y$ for every $n \in \mathbb{N}_{0}$ and $x \in X$. Also let $A_{(n)}$ and $A^{(k)}$ denotes the $n$-th row and $k$-th column of the matrix $A$ respectively for $n,k \in \mathbb{N}_{0}$. The following result is important in this sequel.
\begin{lemma} \cite[Proposition 10]{rao1990hahn}  \label{lemma1}
	Let ${A} = (a_{nk})$ be an infinite matrix. Then $ \textbf{A} \in (h:h)$ if and only if
	\begin{enumerate} %[label=(\roman*)]
		\item[(i)] \label{C1}  $ \underset{n\rightarrow \infty}{\lim}a_{nk} = 0 , \forall k \in \mathbb{N}_{0} $,
		\item[(ii)] \label{C2} $ \sum_{n=0}^{\infty}(n+1) |a_{nk} - a_{n+1,k}| < \infty, \forall k \in \mathbb{N}_{0} $,
		\item[(iii)] \label{C3} $ \ \underset{m}{\sup}\ \frac{1}{m+1}  \sum_{n=0}^{\infty}(n+1) |\sum_{k=0}^{m}(a_{nk} - a_{n+1,k})| < \infty $.
	\end{enumerate}
\end{lemma}

%\section{Convolution operator on weighted Hahn sequence spaces} \label{Section 3}
\section{Weighted Hahn space and its dual} \label{Sect3}

For an arbitrary complex sequence $d = (d_{n})$ with $d_{n} \neq 0$ for all $n \in \mathbb{N}_{0}$, the weighted Hahn sequence space $h_d$ was introduced by Goes \cite{goes1972sequences} where
\[
h_{d} = \Bigl\{x \in  \mathbb{C}^{\mathbb{N}_{0}} : \sum_{n=0}^{\infty} \big|d_{n}\big|\cdot \big|x_{n} - x_{n+1}\big| < \infty\Bigr\}\cap c_{0}.
\]
The norm of any $ x \in h_{d} $ is defined as
\[
\| x \|_{h_{d}} = \sum_{n=0}^{\infty} \big|d_{n}\big|\cdot \big|x_{n} - x_{n+1}\big|.
\]
In particular, if $ d_{n} = n+1 $ for all $n \in \mathbb{N}_{0} $, then $ h_{d} $ reduces to the classical Hahn sequence space $h$, which is defined as follows
\[
h = \Biggl\{x \in \mathbb{C}^{\mathbb{N}_{0}} : \sum_{n=0}^{\infty}(n+1) \cdot \big|x_{n} - x_{n+1}\big| < \infty \Biggr\}\cap c_{0}, 
\]
with 
\[
\| x \|_{h} = \sum_{n=0}^{\infty} \big(n+1\big)\cdot \big|x_{n} - x_{n+1}\big|.
\]
In \cite{malkowsky2021compact}, the boundedness and compactness of any operator $T : h_d \to h_d$ has been studied extensively under the assumption on the weight vector $d$ as an unbounded, monotonic increasing sequence of positive real numbers. Indeed, if we allow $d$ to be bounded then $h_d =h_e = bv_0.$ Therefore, throughout this paper, we assume that the weight vectors are unbounded, monotonic increasing sequence of positive real numbers. The authors also proved that $\big(h_{d}, \|  \cdot \|_{h_{d}}\big)$ is a $BK$ space with $AK$ \cite[Proposition 2.1]{malkowsky2021compact}. Hence, $\big(h_{d}, \|  \cdot \|_{h_{d}}\big),$ forms a Banach space.

For any bounded linear operator $T:h_{d} \to h_{d}$ the operator norm of $T$ is given by \cite[\text{Theorem 3.9}]{malkowsky2021compact} 
\begin{align*}
    \|T\|_{\mathcal{B}(h_{d})} = &\ \underset{m}{\sup}\ \Big( \frac{1}{d_{m}} \sum_{n=0}^{\infty} d_{n} \Big| \sum_{k=0}^{m} (a_{nk} - a_{n+1,k}) \Big| \Big).
\end{align*}
The operator norm $\|T\|_{\mathcal{B}(h)}$ of $T \in \mathcal{B}(h)$ can be easily obtained by putting $d_n = n+1$ in the above equality.

The following lemma gives an inclusion relation between weighted Hahn space and weighted null sequence space $c_0(d)$ which is defined by
\[c_0(d) = \{x \in \mathbb{C}^{\mathbb{N}_{0}} : d_{n}|x_{n}| \to 0 \ \mathrm{ as } \ n \to \infty\}.\]
\begin{lemma}
	Let $d= (d_{n})$ be a monotonic increasing, and unbounded sequence of positive real numbers. Then $h_{d} \subseteq c_{0}(d) $.
\end{lemma}		
\begin{proof}
	Let $x \in h_{d}$. Then $\| x \|_{h_{d}} = \sum_{n=0}^{\infty} d_{n} |x_{n} - x_{n+1}| < \infty $, and $x_{n} \to 0$ as $n \to \infty$. For a preassigned positive $\epsilon$ there exists $N \in \mathbb{N}_{0}$ such that $\sum_{n=N}^{\infty}d_{n}\big|x_{n}-x_{n+1}\big| < \epsilon$. Now for $k \in \mathbb{N}$
	\begin{align*}
		d_{N} \big|x_{N} - x_{N+k}\big| 
		& \leq \sum_{n=0}^{k-1}d_{N} \big|x_{N+n} - x_{N+n+1}\big| \\
		& \leq \sum_{n=0}^{k-1}d_{N+n} \big|x_{N+n} - x_{N+n+1}\big| \\
		& = \sum_{n=N}^{N+k-1} d_{n} \big|x_{n} - x_{n+1}\big|.
	\end{align*}
	By letting $k \to \infty,$ we have $d_{N}\big|x_{N}\big| \leq \sum_{n=N}^{\infty}d_{n} \big|x_{n} - x_{n+1}\big| < \epsilon $. Hence for all $m \geq N,\ d_{m}|x_{m}|< \epsilon $. This implies $x \in c_{0}(d)$.
\end{proof}

As $\big(h_{d}, \|  \cdot \|_{h_{d}}\big)$ is a $BK$ space with $AK,$ by Lemma \ref{b-c-dual} the continuous and $\beta$-duals $ h_{d}^* $ and $ h_{d}^{\beta} $ of $ h_{d} $ are norm isomorphic \cite[Proposition 2.3]{malkowsky2021compact}.
The continuous dual of the weighted Hahn sequence space $h_{d}$ is isomorphic to the sequence space $bs_d$ \cite[Proposition 2.3]{malkowsky2021compact} where  
\[
bs_{d} = \Biggl\{ x \in  \mathbb{C}^{\mathbb{N}_{0}} :  \| x  \|_{bs_{d}} = \underset{m}{\sup} \ \frac{1}{d_{m}} \ \Biggl|\sum_{k=0}^{m}x_{k}\Biggr| < \infty \Biggr\}.
\] 
For the Hahn sequence space $h$, the continuous dual is isomorphic to $\sigma_{\infty}$ which is defined as follows:
\[
\sigma_{\infty} = \Biggl\{ x \in \mathbb{C}^{\mathbb{N}_{0}} : \| x  \|_{\sigma_{\infty}} = \underset{m}{\sup}\ \frac{1}{m+1}\Biggl|\sum_{k=0}^{m} x_{k}\Biggr| < \infty\Biggr\}.
\]

In the next few results, we investigate some of the key properties of $bs_d$.
\begin{lemma} \label{bsd-nAK}
	The space $ \big( bs_{d},\ \big\| \cdot  \big\|_{bs_{d}} \big ) $ does not possess the $AK$ property.
\end{lemma}

\begin{proof}
	We know that $d = (d_n)$ is a monotonic increasing, and unbounded sequence of positive real numbers. Consider the sequence $ x =(x_n)$ defined by 
	\[ x_{n} = 
	\begin{cases}
		0, &  \mathrm{if} \  n=0, \\
		d_{n}, & \mathrm{if} \ n=2k+1 \ \mathrm{for} \ k \in \mathbb{N}_{0}, \\
		-d_{n-1}, & \mathrm{if} \ n=2k \ \mathrm{for} \ k \in \mathbb{N}.
	\end{cases} \]
	Then, $ \| x  \|_{bs_{d}} = 1 $. So, $ x \in bs_{d}$. But, 
    \[ \big\|x-x^{[2n]}\big\|_{bs_{d}} = \big\| \big(  0, 0, \cdots, 0, d_{2n+1}, -d_{2n+1}, d_{2n+3}, \cdots \big)  \big\|_{bs_{d}} = 1 \not\to 0  \mbox{ as }  n \to \infty. \] 
    Therefore $bs_{d}$ does not possess the $AK$ property.
\end{proof}

\begin{lemma} \label{bsd-ns} 
	The space $ bs_{d} $ is a non-separable space.
\end{lemma}

\begin{proof}
	We will prove the lemma by showing the space $bs_{d}$ is norm isomorphic to the space $\ell^{\infty}$. Consider the map $\psi: bs_{d} \to \ell^{\infty}$ defined by $y = \psi(x) = \Big(\frac{1}{d_{m}} \big(\sum_{k=0}^{m} x_{k}\big)\Big) $. Then linearity of $\psi$ is obvious. Also, it is easy to prove that $\psi$ is injective. For surjectivity, let $y = \big(y_{0}, y_{1}, y_{2}, \cdots \big) \in \ell^{\infty} $. Now define the sequence $ x = (x_{n}) $ by $x_{n}= d_{n} y_{n} - d_{n-1}y_{n-1} $ for all $n \in \mathbb{N}_{0} $ with $ d_{-1}y_{-1} = 0$. Then, we have 
	\begin{align*}
		\| x \|_{bs_{d}} 
		&= \underset{m}{\sup} \ \frac{1}{d_{m}} \ \Biggl|\sum_{k=0}^{m}x_{k}\Biggr| \\
		&= \underset{m}{\sup} \ \frac{1}{d_{m}} \ \Biggl|\sum_{k=0}^{m}\Big(d_{k} y_{k} - d_{k-1}y_{k-1}\Big)\Biggr| \\
		&= \underset{m}{\sup} \ \frac{1}{d_{m}} \ \bigl|d_{m}y_{m}\bigr| \\
		&= 	\underset{m}{\sup} \ | y_{m} | \\
		&= \| y \|_{\ell^{\infty}}.
	\end{align*}
	Thus, $ \psi $ is surjective and norm preserving, i.e., $ \psi $ is a linear bijection which preserves the norm.
\end{proof}

\begin{proposition}\label{B-beta-S-A-beta}
    Let $A \subseteq B$ be two subsets of $\mathbb{C}^{\mathbb{N}_{0}}$. Then $B^{\beta} \subseteq A^{\beta}$.
\end{proposition}

\begin{proof}
    Let $b = (b_{n}) \in B^{\beta}$. Then for all $x = (x_{n}) \in B, \ bx = (b_{n}x_{n}) \in cs$. As $A \subseteq B$, it follows that $bx \in cs$ for all $x \in A$. Therefore $b \in A^{\beta}$.
\end{proof}

\begin{theorem} \label{beta-dual-bsd}
	The $\beta$-dual of $bs_{d}$ is $h_{d}$, i.e., $ h_{d} = bs_{d}^{\beta} $.
\end{theorem}

\begin{proof}
	Let us consider $ x= (x_{n}), y= (y_{n}) \in   \mathbb{C}^{\mathbb{N}_{0}} $ such that $ x \in h_{d} $ and $ y \in bs_{d} $. We will show that $ x y= (x_{n}y_{n}) \in cs $.\\
	Let $a_{n} = \frac{1}{d_{n}} \sum_{k=0}^{n} y_{k}$. Then $ a= (a_{n}) \in \ell^{\infty}$. We also note that, $ y_{n} = d_{n} a_{n} - d_{n-1} a_{n-1}$ for all $n \in \mathbb{N}$ with $y_{0} = d_{0} a_{0}$. Now, as $x \in h_{d} \subseteq c_{0}(d) $ and $ a \in \ell^{\infty} ,\ d_{n}a_{n}x_{n+1} \to 0\ \text{as}\ n \to \infty$. As $x \in h_{d}$, for $\epsilon >0$ there exists a natural number $N$ such that $\sum_{k=N}^{\infty} d_{k} \big|x_{k}-x_{k+1}\big|< \frac{\epsilon}{2 \|a\|_{\ell^\infty}}$. Without loss of generality also let us assume that $d_{k}\big|a_{k}x_{k+1}\big|< \frac{\epsilon}{2}$ for all $k \geq N$. Now
	\begin{align*}
		\Big|\sum_{k=N+1}^{n} x_{k} y_{k} \Big| 
		&= \Big| \sum_{k=N+1}^{n} x_{k} \big(d_{k} a_{k} - d_{k-1} a_{k-1} \big) \Big| \\ 
		&= \Big| \sum_{k=N+1}^{n} d_{k}a_{k} \big(x_{k} - x_{k+1}\big) + \sum_{k=N+1}^{n} d_{k} a_{k} x_{k+1} - \sum_{k=N+1}^{n}d_{k-1} a_{k-1} x_{k}\Big| \\
		&= \Big| \sum_{k=N+1}^{n} d_{k}a_{k} \big(x_{k} - x_{k+1}\big) \ + \ d_{n}a_{n}x_{n+1} - d_{N}a_{N}x_{N+1} \Big| \\
		& \leq \sum_{k=N+1}^{n} d_{k}\big|a_{k}\big|\big| x_{k} - x_{k+1} \big| + \big| d_{n}a_{n}x_{n+1} \big| + \big| d_{N}a_{N}x_{N+1} \big| \\
		& \leq \big\| a \big\|_{\ell^\infty}\sum_{k=N+1}^{n} d_{k}\big| x_{k} - x_{k+1} \big| + \big| d_{n}a_{n}x_{n+1} \big| + \big| d_{N}a_{N}x_{N+1} \big|.
	\end{align*} 
	Letting $n \to \infty$ we have 
	\begin{align*}
		\Big| \sum_{k=N+1}^{\infty} x_{k} y_{k} \Big| 
		&\leq \big\| a \big\|_{\ell^\infty}\sum_{k=N+1}^{\infty} d_{k}\big| x_{k} - x_{k+1} \big|  + \big| d_{N}a_{N}x_{N+1} \big| \\
		& < \big\| a \big\|_{\ell^\infty} \frac{\epsilon}{2 \| a \|_{\ell^\infty}} + \frac{\epsilon}{2} = \epsilon.
	\end{align*}
	So $ xy \in cs$, which implies $h_{d} \subseteq bs_{d}^\beta$.

    For the converse part, let $S = \Big\{ x \in bs_{d} : z=(z_{n}) = \big( \sum_{k=0}^n x_{k} \big) \in \ell^{\infty} \Big\}$. Then $S \subseteq bs_{d}$. Let $y \in S^{\beta}$. First we will show that $y \in c_{0}$. If possible let $y \notin c_{0}$. Then for $x = (1, -1,1,-1, \cdots) \in S, \ xy= \big(x_{n} y_{n} \big)\notin c_{0} \subseteq cs$, which is a contradiction. So, $y \in c_{0}$. Now for $x \in S \ \text{and} \ y \in S^{\beta}$ we have
    \begin{align*}
		&\sum_{k=0}^{n}x_{k}y_{k} 
		= \sum_{k=0}^{n}(y_{k}- y_{k+1})\sum_{i=0}^{k}x_{i} \ + \ y_{n+1}\sum_{k=0}^{n}x_{k}\\
		\implies & \sum_{k=0}^{n}(y_{k}- y_{k+1})\sum_{i=0}^{k}x_{i} = \sum_{k=0}^{n}x_{k}y_{k} \ - \ y_{n+1}\sum_{k=0}^{n}x_{k} \\
		\implies & \sum_{k=0}^{n}d_{k}(y_{k}- y_{k+1})\frac{1}{d_{k}}\sum_{i=0}^{k}x_{i} =\sum_{k=0}^{n}x_{k}y_{k} \ - \ y_{n+1}\sum_{k=0}^{n}x_{k}.
	\end{align*}
	Since $(x_{k}y_{k}) \in cs$ and $\big(y_{n+1}\sum_{k=0}^{n}x_{k}\big) \in c_{0}$, we have $\Big(d_{k}(y_{k}- y_{k+1})\frac{1}{d_{k}}z_{k}\Big) \in cs$. Now as $x$ varies over all the members of $bs_{d}$ with $z= \big(\sum_{k=0}^{n}x_{k}\big) \in \ell^\infty$, $\Big(\frac{z_{n}}{d_{n}}\Big)$ varies over all the members of $c_{0}$, i.e., $\big(d_{k}(y_{k}- y_{k+1})\big) \in c_{0}^\beta= \ell^1$, which implies, $y= (y_{n}) \in h_{d}$, i.e., $ S^{\beta} \subseteq h_{d}$. As $S \subseteq bs_{d} $, by Proposition \ref{B-beta-S-A-beta}, $bs_{d}^{\beta} \subseteq S^{\beta}$. Therefore, our result follows. 
\end{proof}	

The following result is immediate from Theorem \ref{beta-dual-bsd}.
\begin{corollary}
    The following assertion holds:
    \[h_d^{\beta \beta} = (bs_d)^\beta =h_d.\]
\end{corollary}

\section{Boundedness and compactness of $T_a$} \label{sect4}

In this section, we address the boundedness and compactness of the convolution operator $T_{a}$ defined on weighted Hahn sequence spaces. In addition, we characterize the multipliers $\mathscr{M}(h_d)$ and $\mathscr{M}(bs_d).$ For this purpose, we use the concept of determining set of a $BK$ space.
\begin{definition} \cite[Definition 7.4.2]{wilansky2000summability} 
	Let $X$ be a $BK$ space. A subset $E$ of the set  $c_{00}$ is called a determining \ set for $X$ if $ D(X)= \bigl\{ x \in X : x \in c_{00} \ \mathrm{and} \ \|x\|_{X} \leq 1 \bigr\} $ is the absolutely convex hull of $E$. 
\end{definition}

\begin{lemma} \cite[Theorem 8.3.4]{wilansky2000summability} \label{BddXY}
	Let $X$ be a $BK$ space with $AK$, $E$ be a determining set for $X$, and Y be an $FK$ space. Then $A \in (X:Y)$ if and only if:
	\begin{itemize}
		\item[(i)] \label{BddC1}The columns of $A$ belong to $Y$, that is $A^{(k)}= (a_{nk})_{n=0}^\infty  \in Y$ for all $k \in \mathbb{N}_{0}$.
		\item[(ii)] \label{BddC2} $L(E)$ is a bounded subset of $Y$, where $L(x) = Ax$ for all $x \in X$.
	\end{itemize}  
\end{lemma}
In \cite[Proposition 3.2]{malkowsky2021compact}, it is proved that $E= \{s(d,k) : k \in \mathbb{N}_{0}\}$ is a determining set for $h_{d}$ where $s(d,k) = \frac{1}{d_{k}} e^{[k]}$, where $ e^{[k]}$ denote the $n$-th section of the vector $e$, the vector whose each component is 1. 

\begin{theorem} \label{bdd-hrhs}
	Let $r=(r_{n})$ and $s=(s_{n})$ be two monotonic increasing, and unbounded sequences of positive real numbers. Then $T_{a}: h_{r} \to h_{s}$ is bounded if and only if 
	\begin{equation*}
		\underset{n \to \infty}{\lim}a_{n} = 0, 
	\end{equation*} 
	and 
	\begin{equation} \label{bounded_h2}
		\|T_{a} \|_{\mathcal{B}(h_{r},h_{s})} =  \underset{m}{\sup} \  \frac{1}{r_{m}}\sum_{n=0}^{\infty}s_{n}\big| a_{n-m} - a_{n+1} \big| < \infty.
	\end{equation}
\end{theorem}

\begin{proof}
Let $E = \{y^{(m)}= \frac{1}{r_{m}}e^{[m]} : m \in \mathbb{N}_0\}$ be the determining set for $h_r,$ where $e^{[m]}$ denote the $m$-th section of the vector $e$. From Lemma \ref{BddXY} we have $T_a \in \mathcal{B}(h_r, h_s)$ if and only if
\begin{enumerate}
    \item[(a)] $\underset{m}{\sup} \|T_{a}\big(y^{(m)}\big)\|_{h_{s}} < \infty, \ \text{with} \ T_{a}\big(y^{(m)}\big) \in c_{0} \ \text{for all}\ y^{(m)} \in E,$
    \item[(b)] columns of $[T_a]$ belongs to $h_s.$ 
\end{enumerate}
Now,

\begin{align*}
		\|T_{a}\big(y^{(m)}\big)\|_{h_{s}} 
		&= \sum_{n=0}^{\infty}s_{n}\big|[T_{a}]_{n,*}y^{(m)} - [T_{a}]_{n+1,*}y^{(m)}\big| \\
		&= \sum_{n=0}^{\infty} s_{n} \Bigg|\frac{\sum_{k=0}^{m}\big(a_{n-k} - a_{n+1-k}\big)}{r_{m}}\Bigg| \\
		&= \frac{1}{r_{m}} \sum_{n=0}^{\infty}s_{n} \Big|\sum_{k=0}^{m}\big(a_{n-k} - a_{n+1-k}\big)\Big| \\
		&= \frac{1}{r_{m}} \sum_{n=0}^{\infty}s_{n} \big|a_{n-m} - a_{n+1} \big|.
	\end{align*}
    Also, the $k$-th column $[T_{a}]_{*,k}$ of $[T_a]$ is the $k$ times shift (right) of the zeroth column, where $k \in \mathbb{N}_0$. Hence, it is easy to prove that $T_{a}\big(y^{(m)}\big) \in c_0$ if and only if $a_n \to 0 $ as $n \to \infty.$ Therefore
	\begin{eqnarray*}
	  & & \underset{m}{\sup} \|T_{a}\big(y^{(m)}\big)\|_{h_{s}} < \infty, \ \text{with} \ T_{a}\big(y^{(m)}\big) \in c_{0} \ \text{for all}\ y^{(m)} \in E \\
       & & \iff  \underset{m}{\sup} \frac{1}{r_{m}} \sum_{n=0}^{\infty}s_{n} \big|a_{n-m} - a_{n+1} \big| < \infty \ \text{with} \lim_{n \to \infty} a_n = 0.\   
	\end{eqnarray*}
It remains to show that condition (b) automatically follows from the hypothesis \eqref{bounded_h2}. Indeed,
	\begin{align*}
		\| [T_{a}]_{*,k}\|_{h_{s}} &= \sum_{n=0}^{\infty} s_{n}\big|a_{n-k} - a_{n+1-k}\big| \\
		&=  \sum_{n=0}^{\infty} s_{n}\Big|\sum_{j=0}^{k} \big(a_{n-j} - a_{n+1-j}\big) - \sum_{j=0}^{k-1} \big(a_{n-j} - a_{n+1-j}\big)\Big| \\
		&\leq r_{k} \sum_{n=0}^{\infty} s_{n}\Big|[T_{a}]_{n,*}y^{(k)} - [T_{a}]_{n+1,*}y^{(k)} \Big| + r_{k-1} \sum_{n=0}^{\infty} s_{n}\Big|[T_{a}]_{n,*}y^{(k-1)} - [T_{a}]_{n+1,*}y^{(k-1)}\Big| \\
		&= r_{k}\|T_{a}\big(y^{(k)}\big)\|_{h_{s}} +r_{k-1}\|T_{a}\big(y^{(k-1)}\big)\|_{h_{s}} < \infty.
	\end{align*}
	This proves the result.
\end{proof}
In particular, if $r= s =d$ (say), then we have the following result.
\begin{corollary} \label{bdd-hd}
	The convolution operator $T_{a}:h_{d} \to h_{d}$ is bounded if and only if 
	\begin{equation}
		\underset{n \to \infty}{\lim}a_{n} = 0,
	\end{equation} 
	and 
	\begin{equation}
		\|T_{a} \|_{\mathcal{B}(h_{d})} = \underset{m}{\sup} \  \frac{1}{d_{m}}\sum_{n=0}^{\infty}d_{n}\big| a_{n-m} - a_{n+1} \big| < \infty.
	\end{equation}
	
\end{corollary}

We now turn our attention to the compactness of the convolution operator. It is proved that $T \in \mathcal{B}(h)$ is not compact. However, things may change drastically with the presence of weights. In Theorem \ref{cpt-hrhs}, we obtain a necessary and sufficient condition for $T_a \in \mathcal{B}(h_r, h_s)$ to be compact.

We employ the idea of Hausdorff measure of non-compactness for this purpose. Let $M_{X}$ denote the collection of all bounded subsets of a complete metric space $(X, d)$. The function $\chi: M_{X} \to [0,\infty)$, defined by 
\[\chi(Q)= \inf\bigl\{\epsilon > 0: Q \ \text{has a finite $\epsilon$-net}\bigr\}, \ Q \in M_{X}\]
is called the Hausdorff measure of non-compactness \cite{malkowsky2019advanced}. It is well known that $Q$ is relatively compact if and only if $\chi(Q)=0$ \cite{toledano1997measures}. Basic properties of Hausdorff measure of non-compactness can be found in \cite{banas1980measures,malkowsky2017some,malkowsky2019advanced,toledano1997measures,malkowsky2000introduction,mursaleen2010compactness}.

Let $L \in \mathcal{B}(X,Y)$, where $X,Y$ are Banach spaces. Then the Hausdorff measure of $L$ is denoted by $\|L\|_{\chi}$ and defined by \cite[Theorem 2.23]{malkowsky2000introduction} 
\[\|L\|_{\chi}=\chi\big(L\big(S_{X}\big)\big)\]
where $S_{X}$ denote the unit sphere in $X$. Also we have $L$ is compact if and only if $\ \|L\|_{\chi}=0$ \cite[Theorem 7.11.5]{malkowsky2019advanced}. The following result gives us the bounds of $\chi \big(Q\big)$ for any bounded set in a Banach space $X$.

\begin{theorem}\cite[Theorem 2.23]{malkowsky2000introduction}\label{Theorem10}
	Let $X$ be a Banach space with a Schauder basis $\bigl\{ e_{0}, e_{1}, e_{2}, \cdots \bigr\} $, $Q$ be a bounded subset of $X$, and $P_{n}: X \to X$ the projector onto the linear span of $\bigl\{ e_{0}, e_{1}, e_{2}, \cdots, e_{n}\bigr\}$. Then 
		\[
		\frac{1}{a} \cdot\underset{n \to \infty}{\limsup}\Bigg( \underset{x \in Q}{\sup} \big\| \big(I - P_{n} \big)(x) \big\|\Bigg) \leq \chi \big(Q\big) \leq \underset{n}{\inf}\ \underset{x \in Q}{\sup} \big\| \big(I - P_{n} \big)(x) \big\| \leq \underset{n \to \infty}{\limsup}\Bigg( \underset{x \in Q}{\sup} \big\| \big(I - P_{n} \big)(x) \big\|\Bigg)
			\]
	where $a = \underset{n \to \infty}{\limsup}\ \big\|I-P_{n} \big\|$ is the basis constant of $\bigl\{ e_{0}, e_{1}, e_{2}, \cdots \bigr\} $.
\end{theorem}

\begin{lemma}\label{Lemma4} Let $a = \underset{n \to \infty}{\limsup}\ \big\|I-P_{n} \big\|$ be the basis constant of $\bigl\{ e_{0}, e_{1}, e_{2}, \cdots \bigr\} $ on the weighted Hahn sequence space $h_{d}$ where $d$ is a monotone increasing sequence of positive real numbers. Then $ a = 1+\alpha$, where $0 \leq \alpha \leq 1$.
\end{lemma}

\begin{proof}
Let $ [I-P_{n}] = \big(a_{rs}\big)_{r,s=0}^{\infty} $ be the matrix representation of the operator $ (I- P_{n}) $ with respect to standard ordered basis where $n \in \mathbb{N}_{0}$. Then \[a_{rs}=
	\begin{cases}
		0, & \mathrm{for}\ r \neq s,\\
		0, & \mathrm{for}\ r=s\ \mathrm{and}\ 0  \leq r \leq n, \\
		1, & \mathrm{for}\ r=s\ \mathrm{and}\ r \geq n+1.
	\end{cases} \]
	Now, 
	\begin{align*}
		\big\| I- P_{n} \big\|_{\mathcal{B}(h_{d})}
		& =\underset{m}{\sup} \Bigg(\frac{1}{d_{m}} \sum_{r=0}^{\infty}d_{r} \Big|\sum_{s=0}^{m}\big(a_{rs} - a_{r+1,s}\big)\Big|\Bigg) \\
		& = \underset{m \geq n+1}{\sup} 
		\Bigg( \frac{1}{d_{m}} \sum_{r=0}^{\infty}d_{r} \Big|0 + a_{r,n+1} - a_{r+1,n+1} +  \cdots + a_{r,m} - a_{r+1,m}\Big|\Bigg)\\
		& = \underset{m \geq n+1}{\sup}
		\Bigg( \frac{d_{n}+d_{m}}{d_{m}} \Bigg) \\
		& = 1+\alpha, \ \text{where}\ 0\leq\alpha\leq1.
	\end{align*}
	Therefore, $a = 1+\alpha$.
\end{proof}

Let $A= \big(a_{tk}\big)_{t,k=0}^\infty$ be an infinite matrix and $A^{<n>} = \big(a_{tk}^{<n>}\big)_{t,k=0}^\infty$ denote the matrix where \[a_{tk}^{<n>}=
\begin{cases*}
	0,\ \ \ \ \text{for}\ 0 \leq t \leq n, \\
	a_{tk},\ \text{for}\ t > n.
\end{cases*}\]
Also let $T^{<n>}$ denotes the operator represented by the matrix $A^{<n>}$.

\begin{theorem} \label{cpt-hrhs}
	$T_{a} \in \mathcal{B}(h_{r},h_{s})$ is compact if and only if
	\begin{equation*}
		\underset{n \to \infty}{\limsup}\Big(\underset{m}{\sup}\ 	\Gamma_{(m,n)}^{(r,s)}\Big)=0,
	\end{equation*}
	where
	\begin{equation*}
		\Gamma_{(m,n)}^{(r,s)}= \frac{1}{r_{m}}\Bigg(s_{n} \Big|\sum_{k=0}^{m} a_{n+1-k}\Big| + \sum_{t=n+1}^{\infty}s_{t}\Big|a_{t-m} - a_{t+1}\Big|\Bigg).
	\end{equation*}
\end{theorem}

\begin{proof}
	We know, $T_{a}$ is compact if and only if $\big\|T_{a}\big\|_{\chi} = 0$, where $\big\|T_{a}\big\|_{\chi} = \chi\big(T_{a}\big(S_{h_{r}}\big)\big)$. From Lemma \ref{Lemma4} and using the fact that $T_{a} \in \mathcal{B}(h_{r},h_{s})$ we have
	\[
	\frac{1}{1+\alpha} \cdot\underset{n \to \infty}{\limsup}\Bigg( \underset{y \in T_{a}(S_{h_{r}})}{\sup} \big\| \big(I - P_{n} \big)(y) \big\|\Bigg) \leq \chi \big(T_{a}(S_{h_{r}})\big)= \big\| T_{a}\big\|_{\chi} \leq \underset{n \to \infty}{\limsup}\Bigg( \underset{y \in T_{a}(S_{h_{r}})}{\sup} \big\| \big(I - P_{n} \big)(y) \big\|\Bigg)
	\]
	where $P_{n}$ denotes the projector onto the linear span of $\bigl\{ e_{0}, e_{1}, e_{2}, \cdots, e_{n}\bigr\}$. Let $y \in T_{a}(S_{h_{r}})$. Then $y = T_{a}(x)$ for some $x \in S_{h_{r}}$. Now, 
	\begin{align*}
		&\underset{y \in T_{a}(S_{h_{r}})}{\sup} \|\big(I - P_{n}\big)(y)\|_{h_{s}}\\
		&= \underset{x \in S_{h_{r}}}{\sup} \|\big(I - P_{n}\big)(T_{a}(x))\|_{h_{s}} \\
		&= \underset{x \in S_{h_{r}}}{\sup} \| T_{a}^{<n>}(x)\|_{h_{s}}\\
		&= \|T_{a}^{<n>}\|_{\mathcal{B}(h_{r},h_{s})}\\
		&= \|A^{<n>}\|_{(h_{r}:h_{s})} \\
		&= \underset{m}{\sup} \ \frac{1}{r_{m}} \sum_{t=0}^{\infty}s_{t} \Big|\sum_{k=0}^{m}\Big(a_{t,k}^{<n>} - a_{t+1,k}^{<n>}\Big)\Big| \\
		&= \underset{m}{\sup} \ \frac{1}{r_{m}} \Bigg(s_{n}\Big|\sum_{k=0}^{m} a_{n+1-k}\Big| + \sum_{t=n+1}^{\infty}s_{t}\Big|\sum_{k=0}^{m}\Big(a_{t-k} - a_{t+1-k}\Big)\Big|\Bigg) \\
		&= \underset{m}{\sup} \ \frac{1}{r_{m}} \Bigg(s_{n}\Big|\sum_{k=0}^{m} a_{n+1-k}\Big| + \sum_{t=n+1}^{\infty}s_{t}\Big|a_{t-m} - a_{t+1}\Big|\Bigg) \\
		&= \underset{m}{\sup} \ \Gamma_{(m,n)}^{(r,s)}.
	\end{align*}
	This proves the result.	
\end{proof}

\begin{corollary} \label{cpt-h}
	The convolution operator $T_{a}\in \mathcal{B}(h_{d})$ is not a compact operator.
\end{corollary}
\begin{proof} 
	For $r=s=d$, from Theorem \ref{cpt-hrhs} we have, $T_{a}$ is compact on $h_{d}$ if and only if $\underset{n \to \infty}{\limsup}\Big(\underset{m}{\sup} \ \Gamma_{(m,n)}^{(d)}\Big)=0$, where 
	\begin{equation*}
		\Gamma_{(m,n)}^{(d)}= \frac{1}{d_{m}}\Bigg(d_{n} \Big|\sum_{k=0}^{m} a_{n+1-k}\Big| + \sum_{t=n+1}^{\infty}d_{t}\Big|a_{t-m} - a_{t+1}\Big|\Bigg).
	\end{equation*}
	  From Corollary \ref{bdd-hd}, we have $|a_{n}| \to 0$ as $n \to \infty$. Also, \[\underset{m}{\sup} \ \Gamma_{(m,n)}^{(d)} \geq 
	\frac{1}{d_{k}}\sum_{t=n+1}^{\infty}d_{t}\big|a_{t-k}-a_{t+1}\big|,\] for all $k\in \mathbb{N}_{0}$. Choosing $k=n+1$, we have
	\begin{align*}
		\underset{m}{\sup} \ \Gamma_{(m,n)}^{(d)} &\geq\frac{1}{d_{n+1}}\sum_{t=n+1}^{\infty} d_{t} \big|a_{t-n-1} - a_{t+1}\big| \\
		&\geq \frac{1}{d_{n+1}} d_{n+1} \big|a_{0}-a_{n+2}\big|\\
		&\geq \big(|a_{0}| - |a_{n+2}|\big).
	\end{align*}
	This shows that $\underset{n\to \infty}{\limsup}\Big(\underset{m}{\sup} \ \Gamma_{(m,n)}^{(d)}\Big) \geq |a_{0}| \not=0$. Therefore, $T_{a}$ is not a compact operator.
\end{proof}	
%	We can now focus on characterizing the set $\mathscr{M}(h_{d})$, the multiplier of $h_{d}$

\section{The multiplier algebras $\mathscr{M}(h_{d})$ and $\mathscr{M}(bs_{d})$} \label{sect5}

This section is devoted to characterize the multipliers $\mathscr{M}(h_{d})$ and $\mathscr{M}(bs_{d}).$ Recall that for any sequence space $X$, the multiplier of $X,$ denoted by $\mathscr{M}(X),$ is defined by 
\[
\mathscr{M}(X) = \Bigl\{a \in \mathbb{C}^{\mathbb{N}_{0}} : a*x \in X \ \text{for all} \ x \in X \Bigr\}.
\]
As $h_d$ and $bs_d$ are Banach spaces, the following relations follow from the closed graph theorem
\begin{enumerate}
    \item[(i)] $T_a \in \mathcal{B}(h_d) \iff a \in \mathscr{M}(h_d),$
    \item[(ii)] $T_a \in \mathcal{B}(bs_d) \iff a \in \mathscr{M}(bs_d).$
\end{enumerate}
It is also easy to verify that $ \mathscr{M}(h_d),$ and $ \mathscr{M}(bs_d)$ form algebras with respect to the convolution operation. In this regard, we have the following lemma.

\begin{lemma} \label{MHd}	
	The multiplier of the weighted Hahn sequence space $h_{d}$ is contained in $h_{d}$, i.e.,	$\mathscr{M}(h_{d}) \subseteq h_{d}$.
\end{lemma}

\begin{proof}
Let $a \in \mathscr{M}(h_{d}) $. Then, $T_{a}(x)= a*x \in h_{d}$ for all $x \in h_{d}$. Choose $x = e_{0}$. Then,
\begin{equation*}
	T_{a}(e_{0})= a*e_{0} = a \in h_{d}.
\end{equation*}
Therefore, $\mathscr{M}(h_{d}) \subseteq h_{d}$.
\end{proof}

The nature of the inclusion relation $\mathscr{M}(h_d) \subseteq h_{d}$ depends on the weight sequence $d$. The following example and Theorem \ref{m_h} provide important characterization of $\mathscr{M}(h_{d})$.
\begin{remark}
It is not necessary that $\mathscr{M}(h_{d})= h_{d} $ holds. Consider $d_{n} = (n+1)^{n+1} $. For $e_{1} = \big(0, 1, 0, 0, \cdots \big),\  \|e_{1}\|_{h_{d}}= \sum_{n=0}^{\infty} d_{n}\big|a_{n} - a_{n+1}\big|= d_{0}+ d_{1} = 1+4 = 5 < \infty $. However,
\begin{align*}
	\big\|  T_{e_{1}} \big\|_{\mathcal{B}(h_{d})} 
		&= \underset{m}{\sup} \frac{1}{d_{m}} \sum_{n=0}^{\infty} d_{n} \big|a_{n-m} - a_{n+1}\big| \\
		&= \underset{m}{\sup} \frac{1}{d_{m}} \Big(d_{0} + d_{m+1} \Big) \\
		&= \underset{m}{\sup} \frac{1}{(m+1)^{m+1}} \Big(1 + (m+2)^{m+2}\Big) \\
		&= \infty.
	\end{align*}
	Hence, $ \mathscr{M}(h_{d}) \neq h_{d} $.
\end{remark}
However, with some particular choice of the weight vector $d$ the relation $ \mathscr{M}(h_{d}) = h_{d} $ hold good. In the next result we construct such class of weight vectors.

\begin{theorem}\label{mult-hd-rn}
    Let $r>1$ be any real number and $d_{n} = r^{n}$, for all $n \in \mathbb{N}_{0}$. Then $\mathscr{M}(h_{d}) = h_{d}$.
\end{theorem}

\begin{proof}
    We already know $\mathscr{M}(h_{d}) \subseteq h_{d}$. To prove the reverse inclusion $h_{d} \subseteq \mathscr{M}(h_{d})$, let $a \in h_{d}$. We prove that $T_{a} \in \mathcal{B}(h_{d})$. As $a \in h_{d}, \ \sum_{n=0}^{\infty} r^{n} |a_{n}-a_{n+1}| < \infty$. Also, as $a \in h_{d} \subseteq c_{0},$ we have 
    \[\|a\|_{{\ell^\infty}} = \sup_{n\in \mathbb{N}_0} |a_n|< \infty.\] 
    We need to show $\underset{m}{\sup} \frac{1}{r^{m}}\sum_{n=0}^{\infty} r^{n} |a_{n-m} - a_{n+1}| < \infty$. For $m \in \mathbb{N}_0,$
    \begin{align*}
        &\frac{1}{r^{m}} \sum_{n=0}^{\infty} r^{n}|a_{n-m} - a_{n+1}|\\
        = & \frac{1}{r^{m}} \sum_{n=0}^{m-1} r^{n} |a_{n+1}| + \frac{1}{r^{m}} \sum_{n=0}^{\infty} r^{n+m} |a_{n} - a_{n+m+1}| \\
        \leq & \frac{1}{r^{m}} \sum_{n=0}^{m-1} r^{n} |a_{n+1}| + \sum_{n=0}^{\infty} r^{n} \left( \sum_{p=0}^m |a_{n+p} - a_{n+p+1}| \right) \\
        \leq & \frac{1}{r^{m}} \sum_{n=0}^{m-1} r^{n} |a_{n+1}| + \left(\sum_{n=0}^{\infty} r^{n} |a_{n} - a_{n+1}| + \frac{1}{r} \sum_{n=0}^{\infty} r^{n+1} |a_{n+1} - a_{n+2}| + \cdots + \frac{1}{r^{m}} \sum_{n=0}^{\infty} r^{n+m} |a_{n+m} - a_{n+m+1}|\right) \\
        \leq & \|a\|_{{\ell^\infty}} \sum_{n=1}^{m} \frac{1}{r^{n}} + \|a\|_{h_{d}} \sum_{k=0}^{m}\frac{1}{r^{k}}.
    \end{align*}
As $r>1$, the series $\sum_{n=0}^{\infty}\frac{1}{r^{n}}$ is convergent. Therefore, $\underset{m}{\sup} \frac{1}{r^{m}} \sum_{n=0}^{m} r^{n}|a_{n-m} - a_{n+1}|< \infty$. Hence, the desired result follows.
\end{proof}

\begin{theorem}\label{m_h}
	The multiplier of the Hahn sequence space $h$ is $h$ itself, i.e., $ \mathscr{M}(h) = h $.
\end{theorem}		
\begin{proof}
	It is enough to show that $h \subseteq \mathscr{M}(h)$. Let us assume that $ a \in h $. Then $a_{n} \to 0$ as $n \to \infty$. So, condition $(i)$ of Lemma \ref{lemma1} is satisfied. As, $a \in h,\ \sum_{n=0}^{\infty}(n+1)|a_{n} - a_{n+1}|< \infty$. This implies that $\sum_{n=0}^{\infty}(n+k+1)|a_{n} - a_{n+1}|< \infty$ for every $k \in \mathbb{N}_{0}$. %(by comparison test, limit form)
	So, condition $(ii)$ of Lemma \ref{lemma1} is also satisfied. For the condition $(iii)$ of Lemma \ref{lemma1}, let
	\begin{align*}
		t_{m} 
		& = \frac{1}{m+1} \sum_{n=0}^{\infty}(n+1) \Big|\sum_{k=0}^{m} (a_{nk} - a_{n+1,k}) \Big| \\
		& = \frac{1}{m+1} \sum_{n=0}^{\infty}(n+1) \Big|\sum_{k=0}^{m} (a_{n-k} - a_{n+1-k}) \Big| \\
		&= \frac{1}{m+1} \sum_{n=0}^{\infty}(n+1) \Big| a_{n-m} - a_{n+1} \Big|
	\end{align*}
	where $a_{-n} = 0$ for all $n \in \mathbb{N} $. As $a \in h$, $ \sum_{n=0}^{ \infty} |a_{n}| < \infty $. So $ (n+1) |a_{n+1}| \to 0$ as $n \to \infty$. Therefore, $ \frac{\sum_{n=0}^{m-1}(n+1)|a_{n+1}|}{m} \to 0 \ \text{as} \ j \ \to \infty $, which implies
	\begin{equation}
		\underset{m}{\sup} \  \frac{1}{m+1} \sum_{n=0}^{m-1} (n+1)|a_{n+1}| < \infty .
	\end{equation}
	As $ a \in h $, we have 
	\begin{align}
		t_{0} = \sum_{n=0}^{\infty}(n+1) |a_{n} -a_{n+1}| < \infty .
	\end{align}
	Let $ m \geq 1 $. Then,
	\begin{align*}
		t_{m} 
		&= \frac{1}{m+1} \sum_{n=0}^{\infty} (n+1)| a_{n-m} - a_{n+1} |\\
		&= \frac{1}{m+1} \sum_{n=0}^{m-1}(n+1)|a_{n+1}| \ + \ \frac{1}{m+1} \sum_{n=m}^{\infty}(n+1) |a_{n-m} - a_{n+1} |\\
		&= \frac{1}{m+1} \sum_{n=0}^{m-1} (n+1)|a_{n+1}| \ + \ \frac{1}{m+1} \sum_{n=0}^{\infty} (n+m+1) | a_{n} - a_{n+m+1}|\\
		&\leq \frac{1}{m+1} \sum_{n=0}^{m-1} (n+1) |a_{n+1}| \ + \ \frac{m}{m+1} \sum_{n=0}^{ \infty} \big(|a_{n}| + |a_{n+m+1}|\big) \ + \ \frac{1}{m+1} \sum_{n=0}^{ \infty}(n+1) | a_{n} - a_{n+m+1}|\\
		&\leq \frac{1}{m+1} \sum_{n=0}^{m-1} (n+1) |a_{n+1}|\ + \ \frac{m}{m+1} \sum_{n=0}^{\infty} |a_{n+m+1}| \ + \ \frac{m}{m+1} \sum_{n=0}^{\infty}|a_{n}| \ +\\
		&\frac{1}{m+1} \sum_{n=0}^{\infty}(n+1)\Bigl\{ |a_{n} - a_{n+1}|\  +  \ |a_{n+1} - a_{n+2}| \ + \ ... \ + \ |a_{n+m} - a_{n+m+1}|\Bigr\}\\
		&= \frac{1}{m+1} \sum_{n=0}^{m-1} (n+1) |a_{n+1}| \ + \ \frac{m}{m+1} \sum_{n=0}^{\infty}|a_{n+m+1}| \ + \ \frac{m}{m+1} \sum_{n=0}^{\infty} |a_{n}| \ + \\
		&\frac{1}{m+1} \sum_{n=0}^{\infty}(n+1)|a_{n} - a_{n+1}| \ +...+ \frac{1}{m+1} \sum_{n=0}^{\infty}(n+1)|a_{n+m} - a_{n+m+1}|\\
		&\leq\frac{1}{m+1} \sum_{n=0}^{m-1}(n+1)|a_{n+1}|\ + \ \frac{2m}{m+1}\sum_{n=0}^{\infty} |a_{n}|\ + \ \frac{1}{m+1}\sum_{n=0}^{\infty} (n+1)|a_{n} - a_{n+1}|\ +...\\
		&+ \frac{1}{m+1} \sum_{n=0}^{\infty} (n+m+1) |a_{n+m} - a_{n+m+1}|\\
		&\leq \frac{1}{m+1} \sum_{n=0}^{m-1}(n+1)|a_{n+1}| \ +\ \frac{2m}{m+1}\sum_{n=0}^{\infty}|a_{n}|\ + \ \frac{m+1}{m+1} \|a\|_{h} .
	\end{align*}
	Using (8) and (9) we have $ \underset{m}{\sup}\ \frac{1}{m+1} \sum_{n=0}^{\infty} (n+1) \big|a_{n-m} - a_{n+1}\big| < \infty $. Therefore, $ h \subseteq \mathscr{M}(h)$.
\end{proof}

%\subsection{The multiplier $\mathscr{M}(bs_d)$}

 Now we characterize the multiplier of the space $bs_d.$ 
\begin{theorem}  \label{M-bsd-l1}
	Let $ d= (d_{n}) $ be a monotonic increasing, and unbounded sequence of positive real numbers. Then the multiplier $  \mathscr{M}(bs_{d}) $ of $bs_{d}$ contains $  \ell^1 $, i.e., $\mathscr{M}(bs_{d}) \supseteq \ell^1$.
\end{theorem}	

\begin{proof}
	Let $  a \in \ell^1 $. Want to show that, for all $ x \in bs_{d},\ a*x \in bs_{d} $.
	\begin{align*}
		&\ \underset{m}{\sup} \ \frac{1}{d_{m}} \ \Big|  \sum_{k=0}^{m} (a*x)_{k} \Big| \\
		= &\ \underset{m}{\sup} \ \frac{1}{d_{m}} \ \Big| a_{0} \sum_{k=0}^{m} x_{k} \  + a_{1} \sum_{k=0}^{m-1} x_{k} \ + \cdots + a_{m}x_{0} \Big| \\
		\leq &\ \underset{m}{\sup} \ \Bigg(\big|a_{0}\big|\ \frac{1}{d_{m}} \Big|\sum_{k=0}^{m} x_{k}\Big| + \big|a_{1}\big|\ \frac{1}{d_{m}} \Big|\sum_{k=0}^{m-1} x_{k}\Big| + \cdots + \big|a_{m}\big|\ \frac{1}{d_{m}} \big|x_{0}\big|\Bigg) \\
		= &\ \underset{m}{\sup} \ \Bigg(\big|a_{0}\big|\ \frac{1}{d_{m}} \Big|\sum_{k=0}^{m} x_{k}\Big| + \frac{d_{m-1}}{d_{m}}\big|a_{1}\big|\ \frac{1}{d_{m-1}} \Big|\sum_{k=0}^{m-1} x_{k}\Big| + \cdots + \frac{d_{0}}{d_{m}} \big|a_{m}\big|\ \frac{1}{d_{0}} \big|x_{0}\big|\Bigg) \\
		\leq &\ \underset{m}{\sup} \ \Bigg(\big|a_{0}\big|\ \frac{1}{d_{m}} \Big|\sum_{k=0}^{m} x_{k}\Big| + \big|a_{1}\big|\ \frac{1}{d_{m-1}} \Big|\sum_{k=0}^{m-1} x_{k}\Big| + \cdots + \big|a_{m}\big|\ \frac{1}{d_{0}} \big|x_{0}\big|\Bigg) \\
		\leq &\ \underset{m}{\sup} \Bigg( \big\| x \big\|_{bs_{d}}  \ \sum_{k=0}^{m} \big|a_{k}\big|  \Bigg) \\
		  = &\ \big\| x \big\|_{bs_{d}} \big\| a\big\|_{\ell^{1}}  < \infty.
	\end{align*}
	This proves the desired result.
\end{proof}	

The following example shows that for certain weights $d,$ we have $\mathscr{M}(bs_{d}) \supsetneq \ell^1.$

\begin{example}
	Let $d_{n}= 5^n, \ n \in \mathbb{N}_{0}$ and $a= e= (1,1,1, \cdots)$. Then $a \notin \ell^1$. We will show that $a \in \mathscr{M}(bs_{d})$. Let $x \in bs_{d}$. Then
	\begin{align*}
		\ \big\| a*x \big\|_{bs_{d}} =& \ \underset{n}{\sup}\ \frac{1}{d_{n}} \Big|\sum_{k=0}^{n} (a*x)_{k}\Big| \\
		=& \ \underset{n}{\sup}\ \frac{1}{d_{n}} \Big|(a_{0}x_{0})+(a_{1}x_{0}+a_{0}x_{1})+ \cdots+ (a_{n}x_{0}+a_{n-1}x_{1}+\cdots+a_{0}x_{n})\Big| \\
		= & \ \underset{n}{\sup}\ \frac{1}{d_{n}} \Big|a_{0}\sum_{k=0}^{n}x_{k} \ + \ a_{1}\sum_{k=0}^{n-1}x_{k} \ + \ \cdots \ + a_{n}x_{0}\Big|\\
		= & \ \underset{n}{\sup}\ \frac{1}{5^{n}} \Big|\sum_{k=0}^{n}x_{k} \ + \ \sum_{k=0}^{n-1}x_{k} \ + \ \cdots \ + x_{0}\Big|\\
		\leq & \ \underset{n}{\sup} \Biggl\{\frac{1}{5^n}\Big|\sum_{k=0}^{n}x_{k}\Big|+ \frac{1}{5}\frac{1}{5^{n-1}}\Big|\sum_{k=0}^{n-1}x_{k}\Big|\ + \ \cdots \ + \frac{1}{5^n}1 \Big|x_{0}\Big|\Biggr\}\\
		\leq & \ \underset{n}{\sup} \Biggl\{1+ \frac{1}{5} + \cdots+ \frac{1}{5^n}\Biggr\} \big\|x\big\|_{bs_{d}}\\
		= & \ \frac{5}{4} \ \big\|x\big\|_{bs_{d}}.
	\end{align*}
	As $x \in bs_{d}, \ \|a*x\|_{bs_{d}} < \infty$. This implies $a \in \mathscr{M}(bs_{d})$.
\end{example}

In view of the above discussion, the following theorem establishes a sufficient condition for $\mathscr{M}(bs_{d})=\ell^1.$ For this purpose, let $a \in \ell^1 \subseteq \mathscr{M}(bs_{d})$. Then as usual, the convolution operator $T_a : bs_d \to bs_d$ is defined as
\[T_a(x) = a * x,  \ x \in bs_d.\]

\begin{theorem} \label{M-bd_d}
	Let $d= (d_{n}) $ be a monotonic increasing, and unbounded sequence of positive real numbers such that $\underset{l \to \infty}{\lim} \frac{d_{l}}{d_{n+l}} = c \neq 0$ for all $n \in \mathbb{N}_{0}$, then the following relations hold:
    \begin{enumerate}
        \item[(i)] $\mathscr{M}(bs_{d})=\ell^1,$
        \item[(ii)] If $a \in \ell^1$ then $c \|a\|_{\ell^1} \leq \|T_{a}\| \leq \|a\|_{\ell^1}.$
    \end{enumerate}
	
	%		 Then the convolution operator $T_{a}: \mathbb{C}^{\mathbb{N}_{0}} \to  \mathbb{C}^{\mathbb{N}_{0}}$ is continuous from $ bs_{d} $ to $bs_{d} $ if and only if $a \in l^1$. And in that case $ \big\| T_{a} \big\|_{\mathcal{B}(bs_{d})} = \sum_{n=0}^{\infty}|a_{n}|$.
\end{theorem}

\begin{proof}
Since $d= (d_{n}) $ is a monotonic increasing, and unbounded sequence of positive real numbers, $c \in (0, 1].$

(i) \ Form Theorem \ref{M-bsd-l1} it follows that $\mathscr{M}(bs_{d}) \supseteq \ell^1.$ For the converse, let us assume that $a \in \mathscr{M}(bs_{d}).$ This implies $ \big\| T_{a} \big\|_{\mathcal{B}(bs_{d})} < \infty$. Then

	%First, let us assume that $a \in \ell^1$. Then from the proof of Theorem \ref{M-bsd-l1} it follows that $\|T_{a}(x)\|_{bs_{d}} \leq \|x\|_{bs_{d}} \|a\|_{\ell^{1}}$, i.e., $\|T_{a}\|_{\mathcal{B}(bs_{d})} \leq \|a\|_{\ell^{1}}$. For the converse, let us assume that $ \big\| T_{a} \big\|_{\mathcal{B}(bs_{d})} < \infty$. Then
	\begin{align*}
		\infty >  \big\| T_{a} \big\|_{\mathcal{B}(bs_{d})} 
		&= \underset{\|x\| \neq 0}{\sup} \frac{\|a*x\|_{bs_{d}}}{\|x\|_{bs_{d}}} \\
		&= \underset{\|x\| \neq 0}{\sup} \frac{ \underset{m}{\sup} \frac{1}{d_{m}}\Bigl|\sum_{k=0}^{m}(a*x)_{k}\Bigr|} {\underset{m}{\sup} \frac{1}{d_{m}}\Bigl|\sum_{k=0}^{m}x_{k}\Bigr|}.
	\end{align*}
	Let $m$ and $n$ be two fixed non-negative integer and 
    \[x= \Big(0, 0, \cdots, 0, e^{-i \theta_{n}}, e^{-i \theta_{n-1}}- e^{-i \theta_{n}}, \cdots, e^{-i \theta_{0}}-e^{-i \theta_{1}}, 0, 0, \cdots\Big),\] 
    where the first nonzero entry $e^{-i \theta_{n}}$ of $x$ appears in the $m$-th position and $ \theta_{t} = \arg a_{t}; \  0 \leq t \leq n$. Then
	\begin{eqnarray*}
	    x*a &=&\Bigl(0, \cdots,0, a_{0}e^{-i \theta_{n}}, a_{0}\big(e^{-i \theta_{n-1}}-e^{-i \theta_{n}}\big)+a_{1}e^{-i \theta_{n}}, \cdots , a_{0}\bigl(e^{-i \theta_{0}}-e^{-i \theta_{1}}\bigr)+ \cdots \Bigr. \\
        & +& \Bigl. a_{n}e^{-i \theta_{n}},a_{1}\bigl(e^{-i \theta_{0}}-e^{-i \theta_{1}}\bigr)+ \cdots + a_{n+1}e^{-i \theta_{n}}, \cdots\Bigr).
	\end{eqnarray*}
	Then $\|(x*a)\|_{bs_{d}} \geq \frac{1}{d_{m+n}} \sum_{k=0}^{n}|a_{k}|$ and $\|x\|_{bs_{d} } = \frac{1}{d_m}$. So, $ \big\| T_{a} \big\|_{\mathcal{B}(bs_{d})} \geq \frac{d_m}{d_{m+n}} \sum_{k=0}^{n}|a_{k}|$. Letting $m \to \infty$, we have $ \big\| T_{a} \big\|_{\mathcal{B}(bs_{d})} \geq c \cdot \sum_{k=0}^{n}|a_{k}|$ for all $n \in \mathbb{N}_{0}$. Hence, 
    \[ \big\| T_{a} \big\|_{\mathcal{B}(bs_{d})} \geq c \cdot \sum_{k=0}^{\infty}|a_{k}|.\]
    This implies $a \in \ell^1.$

    (ii) If $a \in \ell^1,$ then from the proof of Theorem \ref{M-bsd-l1} it follows that $\|T_{a}(x)\|_{bs_{d}} \leq \|x\|_{bs_{d}} \|a\|_{\ell^{1}}$, i.e., $\|T_{a}\|_{\mathcal{B}(bs_{d})} \leq \|a\|_{\ell^{1}}$. Also, from the above proof we have $\big\| T_{a} \big\|_{\mathcal{B}(bs_{d})} \geq c \cdot \sum_{k=0}^{\infty}|a_{k}|.$ Hence the result follows.
    
   % Therefore, if $\underset{l \to \infty}{\lim} \frac{d_{l}}{d_{n+l}} = c$ for some $c \in(0, 1]$ and for all $n \in \mathbb{N}_{0}$, then $ T_{a}$ is a bounded linear operator on $bs_{d}$ if and only if $a \in \ell^1$. And if $c=1$, then $\|T_{a}\|_{\mathcal{B}(bs_{d})} = \sum_{i=0}^{\infty}|a_{i}|$.	
\end{proof}	

\begin{remark}
    In Theorem \ref{M-bd_d}, if $\underset{l \to \infty}{\lim} \frac{d_{l}}{d_{n+l}} = c =1$ then form the relation $(ii)$ of the above theorem it follows that $\|T_{a}\|_{\mathcal{B}(bs_{d})} = \|a\|_{\ell^1}.$
\end{remark}

\begin{corollary}
	The multiplier of $\sigma_{\infty}$ is $\ell^1$.
\end{corollary}	
\begin{proof}
	Here $d_{n}= n+1$ for all $n\in \mathbb{N}_{0}$. So, in this case $c= 1$. Therefore, by Theorem \ref{M-bd_d}, $\mathscr{M}(\sigma_{\infty})= \ell^1$. 
\end{proof}

\section{Spectrum and fine spectrum of $T_a \in \mathcal{B}(h_{d})$} \label{sect6}

We now shift our focus towards the spectrum of $T_{a} \in \mathcal{B}(h_{d})$. For this purpose, we find $\sigma(R,h_{d})$ where $R$ denotes the right shift operator. Let $\mathbb{D}_r$ denote the open disk centered at origin and radius $r$ with $\mathbb{D}_1 =\mathbb{D}$. 
\begin{theorem}
	Let $d= (d_{n})$ be a monotonic increasing, and unbounded sequence of positive real numbers such that $\underset{n \to \infty}{\lim}\frac{d_{n+1}}{d_{n}}=r < \infty$. Then 
    \begin{itemize}
        \item[(i)] $R \in \mathcal{B}(h_{d}),$
        \item[(ii)] $\sigma(R, h_d) = \overline{\mathbb{D}}_{r}.$ 
    \end{itemize}
    
\end{theorem}
\begin{proof}
	Since $d=(d_{n})$ is a monotonic increasing, and unbounded sequence of positive real numbers and $\underset{n \to \infty}{\lim}\frac{d_{n+1}}{d_{n}}=r < \infty$, it follows, $r\geq1$, and 
   $\underset{n}{\sup}\frac{d_{n+1}}{d_{n}}< \infty.$ Let $\underset{n}{\sup}\frac{d_{n+1}}{d_{n}}=c$. Then, 
	\begin{align*}
		\big\|R(x)\big\|_{h_{d}} 
		&= \big\|\big(0, x_{0}, x_{1}, x_{2}, \cdots \big)\big\|_{h_{d}} \\
		&= \sum_{n=0}^{\infty}d_{n}\big|x_{n-1} - x_{n}\big|, \ \text{where}\ x_{-1}=0\\
		&= d_{0}\big|x_{0}\big| + \sum_{n=1}^{\infty} d_{n} \big|x_{n-1} - x_{n}\big|\\
		&= d_{0}\big|x_{0}\big| + \sum_{n=0}^{\infty} d_{n+1} \big|x_{n} - x_{n+1}\big|\\
		& \leq d_{0}\big|x_{0}\big| +  \sum_{n=0}^{\infty} c \cdot d_{n} \big|x_{n} - x_{n+1}\big|\\
		&\leq \|x\|_{h_{d}} + c \cdot  \|x\|_{h_{d}}\\
		&= (1+c)  \|x\|_{h_{d}}.
	\end{align*}
	Hence,  $\|R\|_{\mathcal{B}(h_{d})} \leq 1+c$, and consequently $R \in \mathcal{B}(h_{d})$. 
    
    For the second part, let $ \lambda \in \mathbb{C}$ such that $ (R - \lambda I)$ is invertible, i.e., $\lambda \in \rho(R,h_{d})$. Then the inverse $ (R - \lambda T)^{-1} $ has the following matrix representation.

	\begin{align*}
		B =
		\begin{bmatrix}
			-\frac{1}{\lambda} & 0 & 0 & 0 & \cdots \\
			-\frac{1}{\lambda^2} & -\frac{1}{\lambda} & 0 & 0 & \cdots \\
			-\frac{1}{\lambda^3} & -\frac{1}{\lambda^2} & -\frac{1}{\lambda} & 0 & \cdots \\
			\vdots & \vdots & \vdots &\vdots & \ddots 
		\end{bmatrix} .
	\end{align*}

	It is obvious from the matrix representation that a necessary condition for $(R-\lambda I)^{-1}$ to be bounded on $h_{d}$ is $|\lambda|> 1$ [(2), Corollary \ref{bdd-hd}]. Therefore, $\sum_{n=0}^{\infty} \frac{1}{|\lambda|^n} < \infty$. Let $(b_{n}) = (-\frac{1}{\lambda^{n+1}})$. Now we have
	\begin{align*}
		& \ \big\|(R- \lambda I)^{-1}\|_{\mathcal{B}(h_{d})} \\
		= &\ \underset{m}{\sup} \frac{1}{d_{m}} \sum_{n=0}^{\infty} d_{n} \big|b_{n-m} - b_{n+1}\big|\\
		= &\ \underset{m}{\sup} \Biggl\{\frac{1}{d_{m}} \sum_{n=0}^{m-1}d_{n} \big|b_{n+1}\big| \ + \ \frac{1}{d_{m}} \sum_{n=m}^{\infty} d_{n}\big|b_{n-m}- b_{n+1}\big|\Biggr\}\\
		= &\ \underset{m}{\sup} \Biggl\{\frac{1}{d_{m}} \sum_{n=0}^{m-1}d_{n} \big|b_{n+1}\big| \ + \ \frac{1}{d_{m}} \sum_{n=0}^{\infty} d_{n+m}\big|b_{n}- b_{n+m+1}\big|\Biggr\}\\
		= &\ \underset{m}{\sup} \Biggl\{\frac{1}{d_{m}} \sum_{n=0}^{m-1}d_{n} \Bigg|\frac{1}{\lambda^{n+2}}\Bigg| \ + \ \frac{1}{d_{m}} \sum_{n=0}^{\infty} d_{n+m}\Bigg|\frac{1}{\lambda^{n+1}}-\frac{1}{\lambda^{n+m+2}}\Bigg|\Biggr\}\\
		= &\ \underset{m}{\sup} \Biggl\{\frac{1}{|\lambda|^2}\sum_{n=0}^{m-1}\frac{d_{n}}{d_{m}}\frac{1}{|\lambda|^n} \ + \ \sum_{n=0}^{\infty} \frac{d_{n+m}}{d_{m}} \frac{1}{|\lambda|^n} \frac{1}{|\lambda|} \Bigg|1 - \frac{1}{\lambda^{m+1}}\Bigg|\Biggr\}\\
		%				\leq &\ \underset{m}{\sup}\Biggl\{\frac{1}{|\lambda|^2} \sum_{n=0}^{m-1} \frac{d_{n}}{d_{m}}\frac{1}{|\lambda|^n}\Biggr\} \ + \ \underset{m}{\sup}\Biggl\{\frac{1}{|\lambda|} \sum_{n=0}^{\infty} \frac{d_{n+m}}{d_{m}} \frac{1}{|\lambda|^n}\Biggr\}\  \underset{m}{\sup}\Biggl\{\Bigg|1 - \frac{1}{\lambda^{m+1}}\Bigg|\Biggr\}\\
		%				\leq &\ \underset{m}{\sup} \Biggl\{\frac{1}{|\lambda|^2} \sum_{n=0}^{m-1} \frac{1}{|\lambda|^n}\Biggr\} \ + \ \underset{m}{\sup}\Biggl\{\frac{1}{|\lambda|} \sum_{n=0}^{\infty} \frac{d_{n+m}}{d_{m}} \frac{1}{|\lambda|^n}\Biggr\}\  \underset{m}{\sup}\Biggl\{\Bigg|1 - \frac{1}{\lambda^{m+1}}\Bigg|\Biggr\}\\
		%				= &\ \frac{1}{|\lambda|^2} \sum_{n=0}^{\infty} \frac{1}{|\lambda|^n} \ + \ 2 \ \frac{1}{|\lambda|}\underset{m}{\sup} \sum_{n=0}^{\infty} \frac{d_{n+m}}{d_{m}} \frac{1}{|\lambda|^n}
		= &\ \underset{m}{\sup} \Biggl\{\frac{1}{|\lambda|^2} \sum_{n=0}^{m-1} \frac{d_{n}}{d_{m}} \frac{1}{|\lambda|^n} \ + \ \frac{1}{d_{m}} \frac{1}{|\lambda|} \bigg|1 - \frac{1}{\lambda^{m+1}}\bigg| \sum_{n=0}^{\infty} \frac{d_{n+m}}{|\lambda|^n}\Biggr\},
	\end{align*}
	which is finite if $\underset{n \to \infty }{\lim} \frac{d_{n+m+1}}{d_{n+m}} < |\lambda| $, i.e., $r < |\lambda|$ and infinite if $\underset{n \to \infty }{\lim} \frac{d_{n+m+1}}{d_{n+m}} > |\lambda| $, i.e., $r > |\lambda|$. Therefore, if $|\lambda|>r$, then  $(R- \lambda I)^{-1}$ is a bounded linear operator on $h_{d}$, i.e., if $\lambda \in \overline{\mathbb{D}}_{r}^c$, then $\lambda \in \rho(R)$, i.e., if $\lambda \in \sigma(R, h_{d})$, then $\lambda \in \overline{\mathbb{D}}_{r}$, which implies $\sigma(R, h_{d}) \subseteq \overline{\mathbb{D}}_{r}$. Similarly if $r > |\lambda|$, then $(R-\lambda I)^{-1}$ is not bounded, i.e., $\mathbb{D}_{r} \subseteq \sigma(R,h_{d})$. As spectrum is a closed set, so $\overline{\mathbb{D}}_{r} \subseteq \sigma(R, h_{d})$. Therefore, $\sigma(R, h_{d})= \overline{\mathbb{D}}_{r}$.
\end{proof}

Let $\Phi_{a}(z)$ denote the power series $\sum_{n=0}^{\infty} a_{n}z^{n}$ for some $z \in \mathbb{C}$. The following result is crucial to compute the spectrum of $T_a.$

\begin{theorem} \label{T_a-phi-R}
	Let $T_{a} \in \mathcal{B}(h_{d})$. Then $T_{a}= \Phi_{a}(R)= \sum_{n=0}^{\infty}a_{n}R^n$, where $R$ is the right shift operator on $h_{d}$.
	% and $\Phi_{a}(z)=\sum_{n=0}^{\infty}a_{n}z^n,\ z \in \mathbb{C}$.
\end{theorem}

\begin{proof}
	Let $\Phi_{a}^{(n)}(R)=a_{0}I+a_{1}R+\cdots +a_{n}R^{n}$, where $R^k$ denotes $k$-times composition of right shift operator. Then the matrix representation of the operator $T_{a} - \Phi_{a}^{(n)}(R)$ with respect to standard ordered basis $(e_{n})$ is 
	
	\begin{align*} B=
		\begin{bmatrix}
			0 & 0 & 0 & 0 & \cdots\\
			\vdots & \vdots & \vdots & \vdots &\cdots \\
			0&   0  &   0   &   0   &  \cdots \\
			a_{n+1} & 0 & 0 & 0 & \cdots\\
			a_{n+2} & a_{n+1} & 0 & 0 & \cdots\\
			a_{n+3} & a_{n+2} & a_{n+1} & 0 & \cdots\\
			\vdots & \vdots & \vdots & \vdots & \ddots
		\end{bmatrix}.
	\end{align*} 
	The norm of the operator $T_{a} - \Phi_{a}^{(n)}(R)$ is
	\[ \|T_{a} - \Phi_{a}^{(n)}(R)\|_{\mathcal{B}(h_{d})} = \underset{m}{\sup} \frac{1}{d_{m}}\sum_{k=n}^{\infty} d_{k} \big|a_{k-m} - a_{k+1}\big|,\]
	where $a_{t}=0$ for all $t\leq n$. We claim that, $ \underset{m}{\sup} \frac{1}{d_{m}}\sum_{k=n}^{\infty} d_{k} \big|a_{k-m} - a_{k+1}\big| \to 0$ as $n \to \infty$. If possible let, $ \underset{m}{\sup} \frac{1}{d_{m}}\sum_{k=n}^{\infty} d_{k} \big|a_{k-m} - a_{k+1}\big| \not\to 0$ as $n \to \infty$. Then there exists a $c>0$ and a $M \in \mathbb{N}_{0}$ such that $\frac{1}{d_{M}}\sum_{k=n}^{\infty} d_{k} \big|a_{k-M} - a_{k+1}\big| > c$ for all $n \in \mathbb{N}_{0}$, which contradicts the boundedness of $T_{a}$. Therefore $\Phi_{a}^{(n)}(R)$ converges to $T_{a}$ as $n \to \infty$.
\end{proof}

\begin{corollary} \label{s-Ta-hd}
    Let $T_{a}$ be a bounded linear operator on $h_{d}$, where $d= (d_{n})$ is a monotonic increasing, and unbounded sequence of positive real numbers such that $\underset{n \to \infty}{\lim} \frac{d_{n+1}}{d_{n}} =r$ and $\Phi_{a}(z)$ is analytic on some neighbourhood containing $\overline{\mathbb{D}}_{r}$. Then $\sigma(T_{a}, h_{d})= \Phi_{a}(\overline{\mathbb{D}}_{r})$.
\end{corollary}

\begin{proof}
	As $\Phi_{a}(z)$ is analytic on some neighbourhood containing $\overline{\mathbb{D}}_{r}$, our result follows from Theorem \ref{T_a-phi-R} and the spectral mapping theorem.
\end{proof}

In particular, if $d_n = n+1,$ then $h_d = h.$ Also $T_a \in \mathcal{B}(h) \Leftrightarrow a \in h.$ In this case, $\sigma(R,h) = \overline{\mathbb{D}}$. Let $\mathcal{H}(\overline{\mathbb{D}})$ denote the algebra of complex valued functions which are holomorphic in some neighbourhood containing $\overline{\mathbb{D}}.$ Also let $f_a$ indicate that the sequence $a = (a_n)$ be the coefficient of the power series of $f_a.$ Consider the set 
\[\mathcal{S} = \{a = (a_n) : f_a \in \mathcal{H}(\overline{\mathbb{D}})\}.\]

\begin{corollary}
    Let $a \in \mathcal{S}$ and $\Phi_{a}$ denote the power series of $f_a$. Then $\sigma(T_{a}, h)= \Phi_{a}(\overline{\mathbb{D}}).$
\end{corollary}

\begin{proof}
    Clearly, the radius of convergence $\eta_{f_a}$ of $f_a$ satisfies $\eta_{f_a}>1.$ Therefore, its power series coefficients satisfy $|a_n| \leq \frac{C}{\eta_{f_a}^n}$ for all $n \in \mathbb{N}_0$ and some $C > 0.$ This shows that $a \in h$ and consequently $T_a \in \mathcal{B}(h).$ From Theorem \ref{T_a-phi-R}, we have $T_a = \Phi_{a}(R)$ where $\Phi_a$ denotes the power series of $f_a.$ Finally, the result follows from the spectral mapping theorem. 
\end{proof}

\begin{proposition} \label{pts-Ta-hd}
    Let $ a \in \mathscr{M}(h_{d})$. The point spectrum  of $T_{a}$ is empty, i.e. $\sigma_{p}(T_{a}, h_{d}) = \phi $.
\end{proposition}

\begin{proof}
	Let $x\in h_{d}$. Let $\lambda \in \mathbb{C}$ such that $T_{a}(x)=\lambda x $, for some $x(\neq0)\in h_{d}$. Then we have the following system of equations
	\begin{align*}
		a_{0}x_{0}&=\lambda x_{0}&\\
		a_{1}x_{0}+a_{0}x_{1}&=\lambda x_{1}&\\
		a_{2}x_{0}+a_{1}x_{1}+a_{0}x_{2}&=\lambda x_{3}&\\
		\ \ \ \ \ &\ \ \vdots \  \ \ &
	\end{align*}
	Consider the following cases.

	\textbf{Case 1\ $(a_{0}\neq \lambda)$:} In this case $x_{n}=0, \forall n \in\mathbb{N}_{0}$.	Therefore, $\lambda$ is not an eigenvalue of $T_{a}$.

	\textbf{Case 2\ $(a_{0}=\lambda)$:} Let $k$ be the smallest index for which $x_{k}\neq 0$. Then we have
	\begin{align*}
		a_{0}x_{k}&= \lambda x_{k}\\
		a_{1}x_{k}+ a_{0}x_{k+1} &= \lambda x_{k+1}\\
		a_{2}x_{k}+ a_{1}x_{k+1}+ a_{0}x_{k+2} &= \lambda x_{k+2}\\
		\ \  \ &\ \  \vdots 
	\end{align*}
	This gives $a_{n} = 0$ for all $n \in \mathbb{N}$, which is not true. Therefore, $\sigma_{p}(T_{a}, h_{d}) = \phi$.
\end{proof}	

In the following result we find the point spectrum $\sigma_p(T_a^*, h_d^*)$ as well as the residual spectrum $\sigma_r(T_a, h_d).$
\begin{theorem} \label{rs-Ta-hd}
	Let $T_{a}$ be a bounded linear operator on $h_{d}$, where $d= (d_{n})$ is a monotonic increasing, and unbounded sequence of positive real numbers such that $\underset{n \to \infty}{\lim} \frac{d_{n+1}}{d_{n}} =r$ and $\Phi_{a}(z)$ is analytic on some neighbourhood containing $\overline{\mathbb{D}}_{r}$. Then we have the following results. \begin{itemize}
		\item[(i)] If $r=1$, then 
		\[ \sigma_{r}(T_{a}, h_{d}) = \sigma_p(T_a^*, h_d^*) =
		\begin{cases}
			\Phi_{a}(\overline{\mathbb{D}}), \ & \mathrm{if}\ \underset{n}{\sup} \frac{n+1}{d_{n}} < \infty, \\
			\Phi_{a}(\overline{\mathbb{D}}\setminus \{1\}), \ & \mathrm{if}\ \underset{n}{\sup} \frac{n+1}{d_{n}}= \infty. 
		\end{cases}
		\]
		\item[(ii)] If $r>1$, then
		\[ \sigma_{r}(T_{a}, h_{d}) = \sigma_p(T_a^*, h_d^*) =
		\begin{cases}
			\Phi_{a}(\overline{\mathbb{D}}_{r}), & \mathrm{if}\ \underset{n}{\sup} \frac{r^n}{d_{n}} < \infty, \\
			\Phi_{a}(\mathbb{D}_{r}), & \mathrm{if}\ \underset{n}{\sup} \frac{r^n}{d_{n}}= \infty. 
		\end{cases}
		\]
	\end{itemize}
\end{theorem}	

\begin{proof}
Since $T_a \in \mathcal{B}(h_d),$ $T_a$ has dense range if and only if $T_a^*$ is one-one \cite[p. 59]{goldberg2006unbounded}. Due to this fact, we have 
\[\sigma_{r}(T_a, h_d) = \sigma_{p}(T_a^*,h_d^*)\setminus \sigma_{p}(T_a,h_d) = \sigma_{p}(T_a^*,h_d^*).\]

Hence, we need to find the point spectrum $\sigma_{p}(T_a^*,h_d^*).$ Let $ \lambda \in \mathbb{C} $ such that $T_{a}^{*} x = \lambda x $ for some nonzero $x$. Then we have 
	\begin{align*} 
		\begin{bmatrix}
			a_{0} & a_{1} & a_{2} & a_{3} & a_{4} & \cdots \\
			0 & a_{0} & a_{1} & a_{2} & a_{3} & \cdots \\
			0 &   0   & a_{0} & a_{1} & a_{2} & \cdots \\
			0 &   0   &   0   & a_{0} & a_{1} & \cdots \\
			\vdots & \vdots& \vdots& \vdots& \vdots& \ddots 
		\end{bmatrix}
		\begin{bmatrix}
			x_{0} \\
			x_{1} \\
			x_{2} \\
			x_{3} \\
			\vdots
		\end{bmatrix}
		= \lambda
		\begin{bmatrix}
			x_{0} \\
			x_{1} \\
			x_{2} \\
			x_{3} \\
			\vdots
		\end{bmatrix} .
	\end{align*}
	This gives the following system of equations.
	\begin{align*}
		&\sum_{k=0}^{\infty}a_{k}x_{k} = \lambda x_{0}\\
		&\sum_{k=0}^{\infty}a_{k}x_{k+1} = \lambda x_{1}\\
		&\sum_{k=0}^{\infty}a_{k}x_{k+2} = \lambda x_{2}\\ 
		&\sum_{k=0}^{\infty} a_{k}x_{k+3} = \lambda x_{3}\\
		& \hspace{19mm}  \vdots
	\end{align*}
	Let $x_{0} \neq 0$ and $x_{n}= c^n x_{0}$ for some $c \in \mathbb{C}$. Then $x = \big(x_{0}, c x_{0}, c^2 x_{0}, \cdots \big)$. If $x \in bs_{d}$ then $\lambda_{c} = \Phi_{a}(c) = \sum_{n=0}^{\infty}a_{n}c^n$ will be an eigenvalue and $x$ will be the corresponding eigenvector. But $x = \big(x_{0}, c x_{0}, c^2 x_{0}, \cdots \big) \in bs_{d}$ if and only if $\underset{n}{\sup} \ \frac{|x_{0}|}{d_{n}} \Big|1 + c + c^2 + \cdots + c^n\Big| < \infty$. Consider the following two cases.

	\textbf{Case 1:} $\underset{n \to \infty}{\lim }\frac{d_{n+1}}{d_{n}}=r =1$. In this setting, 
    \[\sigma_r(T_a, h_d) = \sigma_p(T^*_a, h^*_d) \setminus \sigma_p(T_a, h_d) = \sigma_p(T^*_a, h^*_d) \subseteq \sigma(T_a, h_d) = \Phi_a(\overline{\mathbb{D}}).\]
    Hence, $|c| \leq 1$ is necessary for $\lambda_c = \Phi_{a}(c) \in \sigma_p(T^*_a, h^*_d).$ If $|c| \leq 1$ with $c \neq 1,$ then 
    
	\begin{align*}
		& \underset{n}{\sup} \ \frac{|x_{0}|}{d_{n}} \big|1 + c + c^2 + \cdots + c^n\big| \\
		=\ & \underset{n}{\sup} \ \frac{|x_{0}|}{d_{n}} \Bigg|\frac{1 - \ c^{n+1}}{1 - c}\Bigg| \\
		\leq\ & \underset{n}{\sup} \ \frac{|x_{0}|}{d_{n}} \frac{2}{|1 - c |} < \infty.
	\end{align*}
	Hence $\lambda_{c} \in \sigma_{p}(T_{a}^{*},bs_{d})$. Also, for $c=1,$ it follows
	\begin{align*}
		&\underset{n}{\sup} \ \frac{|x_{0}|}{d_{n}} \Big|1+ c+ c^2 + \cdots + c^{n}\Big| \\
		= \ & |x_{0}| \ \underset{n}{\sup} \ \frac{n+1}{d_{n}}.
	\end{align*}
    Hence, 
    \[\lambda_1 = \Phi_a(1) \in \sigma_p(T^*_a, h^*_d) \Leftrightarrow \underset{n}{\sup} \ \frac{n+1}{d_{n}} < \infty.\]
    This proves the result for $r=1.$

	\textbf{Case 2:} $\underset{n \to \infty}{\lim }\frac{d_{n+1}}{d_{n}}=r >1$.

In this case, we have $|c| \leq r.$ If $|c| \leq 1,$ then 
\[\underset{n}{\sup} \ \frac{|x_{0}|}{d_{n}} \Big|1 + c + c^2 + c^{n}\Big| \leq  |x_{0}|\  \underset{n}{\sup} \ \frac{n+1}{d_{n}} < \infty,\] 
as the sequence $\big(\frac{n+1}{d_{n}}\big)$ converges to $0$ as $n$ tends to $\infty$. Also, if $1< |c| \leq r$, then
\begin{align*}
		\underset{n}{\sup} \ \frac{|x_{0}|}{d_{n}} \Big|1 + c + c^2 + \cdots\ + c^{n}\Big| & = \ \underset{n}{\sup} \ \frac{|x_{0}|}{d_{n}} \Bigg|\frac{c^{n+1} - 1 }{c - 1}\Bigg| < \infty \\
		& \iff \underset{n}{\sup} \ \frac{|x_{0}|}{d_{n}} \frac{|c|}{|c - 1|} |c|^n < \infty \\
		& \iff \underset{n}{\sup}\ \frac{|c|^n}{d_{n}} < \infty.
	\end{align*}
Clearly, if $|c| < r$ then $\underset{n}{\sup}\ \frac{|c|^n}{d_{n}} < \infty$ and consequently $\lambda_c = \Phi_a(c) \in \sigma_p(T^*_a, h^*_d).$ Also, if $|c| = r$, then $\lambda_{c} \in \sigma_{p}(T_{a}^{*}, bs_{d})$ if and only if $ \underset{n}{\sup}\ \frac{r^n}{d_{n}}$ is finite. Hence, the result follows.
\end{proof}

\begin{corollary} \label{cs-Ta-hd}
	Let $T_{a}$ be a bounded linear operator on $h_{d}$, where $d= (d_{n})$ is a monotonic increasing, and unbounded sequence of positive real numbers such that $\underset{n \to \infty}{\lim} \frac{d_{n+1}}{d_{n}} =r$ and $\Phi_{a}(z)$ is analytic on some neighbourhood containing $\overline{\mathbb{D}}_{r}$. Then we have the following results. \begin{itemize}
		\item[(i)] If $r=1$, then
		\[ \sigma_{c}(T_{a}, h_{d}) =
		\begin{cases}
			\phi, & \mathrm{if } \ \underset{n}{\sup} \frac{n+1}{d_{n}} < \infty, \\
			\Phi_{a}(\overline{\mathbb{D}})\setminus \Phi_{a}(\overline{\mathbb{D}}\setminus\{1\}), & \mathrm{if } \ \underset{n}{\sup} \frac{n+1}{d_{n}}= \infty. 
		\end{cases}
		\]
		\item[(ii)] If $r>1$, then
		\[ \sigma_{c}(T_{a}, h_{d}) =
		\begin{cases}
			\phi, & \mathrm{if}\ \underset{n}{\sup} \frac{r^n}{d_{n}} < \infty, \\
			\Phi_{a}(\overline{\mathbb{D}}_{r}) \setminus \Phi_{a}(\mathbb{D}_{r}), & \mathrm{if}\ \underset{n}{\sup} \frac{r^n}{d_{n}}= \infty. 
		\end{cases}
		\]
	\end{itemize}
\end{corollary}

\begin{proof}
	The result follows from the fact that for a bounded linear operator on a Banach space the fine spectrum forms a disjoint partition.
\end{proof}

We conclude this section by providing some example to supplement our results.

\begin{example} In this example, we illustrate how the residual spectrum $\sigma_r(T_a, h_d)$ and continuous spectrum $\sigma_c(T_a, h_d)$ depends on the weight vector $d$.
	\begin{itemize}

        \item[(i)] Let $d_{n} = n+1$. Then $r = \underset{n \to \infty}{\lim} \frac{d_{n+1}}{d_{n}}=1,$ and $\underset{n}{\sup} \frac{n+1}{d_{n}} = 1 $. If $\Phi_{a}(z)$ is analytic on some neighbourhood of $\overline{\mathbb{D}}$, then  $\sigma_r(T_a, h_d)  = \sigma(T_{a},h) = \Phi_a(\overline{\mathbb{D}})$ and $ \sigma_c(T_a, h_d)  = \phi$. 

        \item[(ii)] Let $d_{n} = \sqrt{n+1}$. Then $r = \underset{n \to \infty}{\lim} \frac{d_{n+1}}{d_{n}}=1,$ and $\underset{n}{\sup} \frac{n+1}{d_{n}}= \infty$. If $\Phi_{a}(z)$ is analytic on some neighbourhood of $\overline{\mathbb{D}}$, then  $\sigma_r(T_a, h_d)  = \Phi_a(\overline{\mathbb{D}} \setminus \{1\})$ and $ \sigma_c(T_a, h_d)  = \Phi_a(\overline{\mathbb{D}}) \setminus \Phi_a(\overline{\mathbb{D}} \setminus \{1\})$.

        \item[(iii)] Let $d_{n}= (n+1) 3^n$. Then $r = \underset{n \to \infty}{\lim} \frac{d_{n+1}}{d_{n}}=3,$ and $\underset{n}{\sup} \frac{r^n}{d_{n}} = 1.$  If $\Phi_{a}(z)$ is analytic on some neighbourhood of $\overline{\mathbb{D}}_{3}$, then $\sigma_r(T_a, h_d)  = \Phi_a(\overline{\mathbb{D}}_3)$ and $ \sigma_c(T_a, h_d)  = \phi$.
        
		\item[(iv)] Let $d_{n} = \frac{3^n}{n+1}$. Then $r = \underset{n \to \infty}{\lim} \frac{d_{n+1}}{d_{n}}=3,$ and $\underset{n}{\sup} \frac{r^n}{d_{n}}= \infty$. If $\Phi_{a}(z)$ is analytic on some neighbourhood of $\overline{\mathbb{D}}_{3}$, then  $\sigma_r(T_a, h_d)  = \Phi_a(\mathbb{D}_3)$ and $\sigma_c(T_a, h_d)  = \Phi_a(\overline{\mathbb{D}}_3) \setminus \Phi_a(\mathbb{D}_3)$. 
	\end{itemize}
\end{example}	

\begin{example}
    Consider the generating sequence $a = (a_n)$ as
    \[a_{n}=\begin{cases}
    1, \ \mathrm{if}\ n=0,1,3,9;\\
    0, \ \mathrm{otherwise}.
\end{cases}\]
Also consider any weight vector $d_n = n+1.$ Then $a \in h = \mathscr{M}(h),$ and $\Phi_a(z) = 1+z+z^3+z^9.$ In this case, we have
\begin{enumerate}
    \item[(i)] $\sigma(T_a, h) = \sigma_r(T_a, h) = \Phi_a(\overline{\mathbb{D}}),$
    \item[(ii)] $\sigma_p(T_a, h) = \sigma_c(T_a, h) = \phi.$
\end{enumerate}
\end{example}

\begin{example}
Consider the generating sequence $b=(b_{n})$ where $b_{n} = \frac{1}{n!}$ for all $n \in \mathbb{N}_{0}$. Also consider any weight vector $d_n = n+1.$ Then $b \in h = \mathscr{M}(h),$ and $\Phi_{b}(z) = e^z$.
In this case, we have
\begin{enumerate}
    \item[(i)] $\sigma(T_b, h) = \sigma_r(T_b, h) = \Phi_b(\overline{\mathbb{D}}),$
    \item[(ii)] $\sigma_p(T_b, h) = \sigma_c(T_b, h) = \phi.$
\end{enumerate}
\end{example}

The figures $\sigma(T_a, h)= \Phi_a(\overline{\mathbb{D}})$ and $\sigma(T_b, h)= \Phi_b(\overline{\mathbb{D}})$ are given below.

\begin{figure}[htbp]
  \centering
  \begin{subfigure}{0.49\textwidth}
    \centering
    \includegraphics[width=\linewidth]{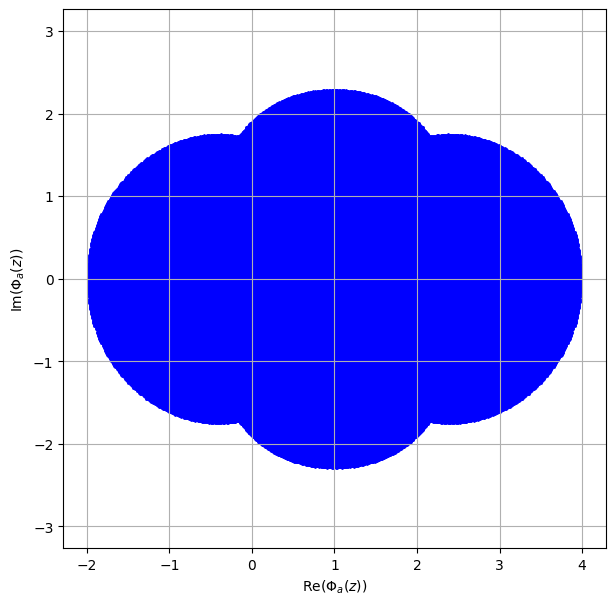}
    \caption{Image of closed unit disk under $\Phi_a(z)$}
    \label{fig:img1}
  \end{subfigure}
  \hfill
  \begin{subfigure}{0.5\textwidth}
    \centering
    \includegraphics[width=\linewidth]{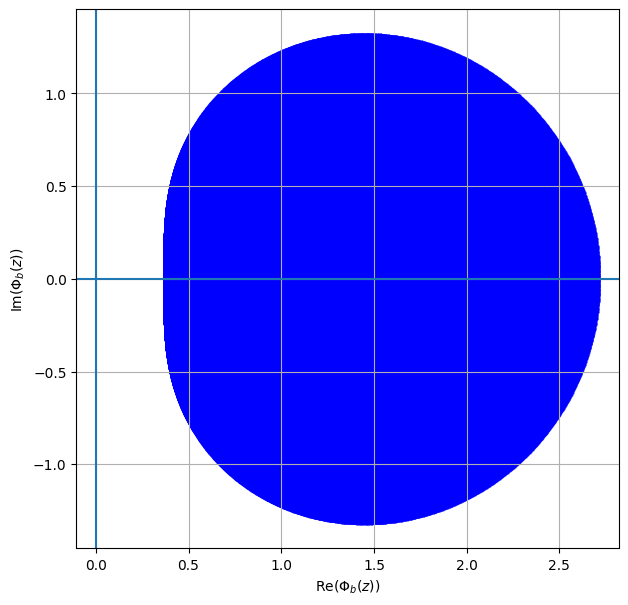}
    \caption{Image of closed unit disk under $\Phi_b(z)$}
    \label{fig:img2}
  \end{subfigure}
  \caption{$\sigma(T_a, h)$ and $\sigma(T_b, h)$}
  \label{fig:twoside}
\end{figure}

\section*{Declarations}
\begin{itemize}
\item Availability of data and materials: Not applicable.
\item Competing interests: The authors declare that they have no competing interests.
\item Funding: Not applicable.
\item Authors' contributions: Both the authors contribute equally to this work.
\item Acknowledgments: Mr. Sayan Saha thanks the UGC, Govt. of India, for financial support in the form of fellowship (NTA Ref. No.: 221610024673).
\end{itemize}

\bibliographystyle{sn-mathphys-num}
\bibliography{Bibliography_All}% common bib file

\end{document}